\documentclass[10pt]{amsart}
\usepackage{amscd,amsmath,amssymb,amsfonts,verbatim}
\usepackage[all, cmtip]{xy}
\usepackage{comment}

\newtheorem{thm}[equation]{Theorem}
\makeatletter
\let\c@subsection\c@equation
\makeatother
\newtheorem{prop}[equation]{Proposition}
\newtheorem{lem}[equation]{Lemma}
\newtheorem{cor}[equation]{Corollary}

\theoremstyle{definition}

\newtheorem{defn}[equation]{Definition}
\theoremstyle{remark}
\newtheorem{remk}[equation]{Remark}
\newtheorem{remks}[equation]{Remarks}

\newtheorem{exm}[equation]{Example}
\newtheorem{exms}[equation]{Examples}
\newtheorem{notat}[equation]{Notation}
\numberwithin{equation}{section}

{\hfill$\square$\end{defn}}
{\hfill$\square$\end{remk}}
{\hfill$\square$\end{remks}}
{\hfill$\square$\end{exm}}
{\hfill$\square$\end{exms}}
{\hfill$\square$\end{notat}}

\newcommand{\thmref}{Theorem~\ref}
\newcommand{\propref}{Proposition~\ref}
\newcommand{\corref}{Corollary~\ref}

\newcommand{\lemref}{Lemma~\ref}

\newcommand{\sC}{{\mathcal C}}

\newcommand{\sF}{{\mathcal F}}

\newcommand{\sH}{{\mathcal H}}

\newcommand{\sK}{{\mathcal K}}
\newcommand{\sL}{{\mathcal L}}
\newcommand{\sM}{{\mathcal M}}

\newcommand{\sO}{{\mathcal O}}

\newcommand{\sR}{{\mathcal R}}
\newcommand{\sS}{{\mathcal S}}

\newcommand{\sZ}{{\mathcal Z}}
\newcommand{\A}{{\mathbb A}}

\renewcommand{\H}{{\mathbb H}}

\renewcommand{\P}{{\mathbb P}}
\newcommand{\Q}{{\mathbb Q}}

\newcommand{\Z}{{\mathbb Z}}

\newcommand{\KGL}{{\rm KGL}}

\newcommand{\CH}{{\rm CH}}
\newcommand{\DM}{{\rm DM}}
\newcommand{\DML}{{{\rm DM}(k, \Lambda)}}

\newcommand{\surj}{\twoheadrightarrow}
\newcommand{\inj}{\hookrightarrow}

\newcommand{\Pic}{{\rm Pic}}

\newcommand{\Hom}{{\rm Hom}}

\newcommand{\Spec}{{\rm Spec \,}}

\newcommand{\divf}{{\rm div}}

\newcommand{\Sch}{{\operatorname{\mathbf{Sch}}}}

\newcommand{\Sm}{{\mathbf{Sm}}}

\newcommand{\ds}{{/\kern-3pt/}}

\newcommand{\un}{\underline}
\newcommand{\ov}{\overline}

\newcommand{\dgn}{{\operatorname{degn}}}
\renewcommand{\dim}{\text{\rm dim}}

\newcommand{\tuborg}{\left\{\begin{array}{ll}}
\newcommand{\sluttuborg}{\end{array}\right.}

\newcommand{\zar}{{\rm zar}}
\newcommand{\nis}{{\rm nis}}
\newcommand{\wt}{\widetilde}

\newcommand{\fnteff}{\mathrm{eff}}

\newcommand{\unsmot}{\mathcal M}

\newcommand{\unsmotcdh}{\mathcal M _{cdh}}

\newcommand{\Tspectra}{Spt(\unsmot)}
\newcommand{\TspectraX}{Spt(\unsmot_{X})}
\newcommand{\Tspectracdh}{Spt(\unsmotcdh)}

\newcommand{\susp}[1]{\Sigma ^{#1}}
\newcommand{\stablehomotopy}{\mathcal{SH}}
\newcommand{\stablehomotopycdh}{\mathcal{SH}_{cdh}}

\newcounter{elno}

\newcounter{elno-abc}   

\newcounter{elno-abc-prime}

\begin{document}
\title{Motivic spectral sequence for relative homotopy $K$-theory}
\author{Amalendu Krishna, Pablo Pelaez}
\address{School of Mathematics, Tata Institute of Fundamental Research,  
1 Homi Bhabha Road, Colaba, Mumbai, India}
\email{amal@math.tifr.res.in}
\address{Instituto de Matem\'aticas, Ciudad Universitaria, 
UNAM, DF 04510, M\'exico.}
\email{pablo.pelaez@im.unam.mx}




\maketitle

\begin{quote}\emph{Abstract.}  
We construct a motivic spectral sequence for the relative homotopy 
invariant $K$-theory of a closed immersion of schemes $D \subset X$.
The $E_2$-terms of this spectral sequence are the $cdh$-hypercohomology
of a complex of equidimensional cycles.

Using this spectral sequence, we obtain a cycle class map from
the relative motivic cohomology group of 0-cycles to the relative
homotopy invariant $K$-theory.
For a smooth scheme $X$ and a divisor $D \subset X$,
we construct a canonical homomorphism from the Chow groups with modulus 
$\CH^i(X|D)$ to the relative motivic cohomology
groups $H^{2i}(X|D, \Z(i))$ appearing in the above spectral sequence.
This map is shown to be an isomorphism when $X$ is affine and 
$i = \dim(X)$.
\end{quote}
\setcounter{tocdepth}{1}
\tableofcontents  


\section{Introduction}\label{sec:Intro}
In this text, we construct an Atiyah-Hirzebruch type spectral sequence for
the relative algebraic $K$-theory of a closed immersion of
smooth schemes, and relate the $E_2$-terms of this spectral sequence
with the known Chow groups with modulus in various cases.
This section provides the background of the problem, a summary of
main results, their statements and outline of proofs.

\subsection{The background}\label{sec:Background}
Since the advent of higher algebraic $K$-theory of rings and schemes by
Quillen, the goal has been to search for tools for
computing these higher $K$-groups. In his seminal work \cite{Bloch},
Bloch introduced the theory of higher Chow groups of schemes.
He showed that these higher Chow groups rationally coincide with the
algebraic $K$-groups of schemes. It was later shown by Bloch-Lichtenbaum 
\cite{BL}, Friedlander-Suslin \cite{FS} and Levine \cite{Levine-3} that there exists an
Atiyah-Hirzebruch type spectral sequence whose $E_2$-terms are Bloch's
higher Chow groups and which abuts to the algebraic $K$-theory of a 
smooth scheme. 

After the introduction of motivic homotopy theory by Voevodsky and his
coauthors, it was subsequently observed by Voevodsky that the motivic 
$T$-spectra in the motivic stable homotopy category over a field
(e.g., the algebraic $K$-theory spectrum $\KGL$) could be understood well via
their slice filtration.  Voevodsky \cite[Chapter~19]{MVW} also showed that    
the motivic cohomology groups appearing in the spectral sequence resulting
from the slice filtration for the algebraic $K$-theory spectrum coincide with 
Bloch's higher Chow groups. Since then, Voevodsky's slice filtration has
become a very powerful tool to compute algebraic $K$-theory of 
smooth schemes.

One immediate question that arises out of the works of Friedlander-Suslin and 
Voevodsky is if it is possible to construct a convergent spectral sequence 
which computes the relative algebraic $K$-theory $K(X,D)$ for a closed 
immersion of
smooth schemes $D \subset X$, and which reduces to the earlier spectral
sequence when $D= \emptyset$.  The first problem that one faces in order to
answer this question is to define a relative motivic cohomology theory which
would constitute the $E_2$-terms of such a spectral sequence.
Based on the theories of additive higher Chow groups by  
Bloch and Esnault \cite{BE} and Chow groups with modulus by Kerz and Saito
\cite{KS}, a theory of higher Chow groups with modulus was 
introduced by Binda and Saito in \cite{BS17}. It is expected that
these higher Chow groups with modulus (or some refined version of them)
constitute the $E_2$-terms of a spectral
sequence which would converge to the relative algebraic $K$-theory
$K(X,D)$ whenever $D$ is an effective Cartier divisor in a smooth scheme $X$
over a field. 


\subsection{Summary of main results}\label{sec:Main-R}
Despite the introduction of higher 
Chow groups with modulus, connecting these groups
to relative algebraic $K$-theory, and in particular,
constructing the desired spectral sequence, remains 
one of the challenging current problems in algebraic $K$-theory.
This paper is an attempt in this direction. Here, we do not
construct a spectral sequence whose $E_2$-terms are the
higher Chow groups with modulus. What we do instead is to expand and feed
the machinery of Voevodsky's slice filtration  into the setting of relative
$K$-theory. 
What results is a strongly convergent spectral sequence abutting to the 
relative
algebraic $K$-theory of a closed immersion of smooth schemes $D \subset X$.
More generally, we show that such a spectral sequence exists for the
relative homotopy invariant $K$-theory $KH(X,D)$ for any
closed immersion of schemes $D \subset X$.  

Given a closed immersion of a divisor $D$ inside a smooth scheme $X$,
we show that the $E_2$-terms of our spectral sequence can be described as
the $cdh$-hypercohomology of a subcomplex of the 
complex of equidimensional cycles of Friedlander-Suslin-Voevodsky     
on a scheme $S_X$. This scheme is obtained from gluing two copies of $X$ 
along $D$. If $X$ is projective, these $E_2$-terms are shown to coincide
with the motivic cohomology with compact support \cite{FV} 
of the complement of $D$ in $X$.

Another aspect of our spectral sequence is its degeneration with rational 
coefficients. An important consequence of this is that it allows us  
to provide a complete description of the rational relative $K$-theory 
of projective schemes in 
terms of the motivic cohomology with compact support. 
Furthermore, it allows us to prove the Grothendieck Riemann-Roch theorem
for relative $K$-theory of smooth schemes (see \thmref{thm:sm-pair}).
Prior to this work, it
was not known if the relative $K$-theory of a smooth pair of projective
schemes could be described in terms of motivic cohomology with compact
support.

Having this spectral sequence in hand, what remains to connect
Chow groups with modulus with relative $K$-theory of a smooth
divisor $D$ inside a smooth scheme $X$ is to show the agreement between
the $E_2$-terms of our spectral sequence and the Chow groups with modulus.
In an attempt in this direction, we construct a canonical homomorphism
from the Chow groups with modulus to the $E_2$-terms of our
spectral sequence. We then show that for 0-cycles, this
map is in fact an isomorphism when $X$ is affine. 
This provides some evidence that
the spectral sequence constructed in this paper might be the answer 
to the question of relating Chow groups with modulus with relative
algebraic $K$-theory of smooth pairs.

\subsection{Statements of main results}\label{sec:Main-R*}
The results we prove can be
summarized as follows. The exact hypothesis of each statement,  
notations and terms used in these results will be explained and made precise 
at appropriate places in this text. 

Let $k$ be a perfect field and let 
$\Lambda$ denote the ring $\Z$ if $k$ admits resolution of singularities or, 
$\Z[\tfrac{1}{p}]$ if $k$ has characteristic $p > 0$. 
For an abelian group $A$, let $A_{\Lambda} = A \otimes_{\Z} \Lambda$.
Given a morphism of schemes $f: D \to X$ over $k$, let $KH(X,D)$ denote the
homotopy fiber of the map of Weibel's homotopy $K$-theory spectra
$f^*: KH(X) \to KH(D)$. Note that $KH(X,D)$ coincides with the relative
algebraic $K$-theory spectrum $K(X,D)$ if $X$ and $D$ are smooth.

Recall from \cite[Lecture~16]{MVW} that the presheaf of abelian groups
$z_{equi}(\A^q_k,0)$ on the category  of smooth schemes over $k$
is defined by letting $z_{equi}(\A^q_k,0)(U)$ be the free abelian 
group generated by the closed and irreducible subschemes 
$Z \subsetneq U \times \A^q_k$ which are dominant and
equidimensional of relative dimension zero over a component of 
$U$.
Let $C_{\ast}z_{equi}(\A^q_k,0)$ denote the chain complex of
presheaves of abelian groups associated, via the Dold-Kan correspondence, to
the simplicial presheaf given by $C_nz_{equi}(\A^q_k,0)(U) 
= z_{equi}(\A^q_k,0)(U \times \Delta^n_k)$. 
Given a smooth scheme $X$ over $k$ and an effective Cartier divisor 
$D \subset X$, let $S_X$ denote the scheme obtained by gluing
two copies of $X$ along $D$ and let $\nabla: S_X \to X$ be the fold
map. We let $\Lambda_{X|D}(q)[2q]$ denote the complex of
sheaves on the $cdh$-site of $X$ given by
$\Lambda_{X|D}(q)[2q] = 
{\rm Ker}((\nabla_*(C_{\ast}z_{equi}(\A^q_k,0)|_{(S_X)_{cdh}}))
\to C_{\ast}z_{equi}(\A^q_k,0)|_{X_{cdh}})$. 

The relative motivic cohomology $H^{a}(X|D, \Lambda(b))$ where
$D\subset X$ is a closed subscheme, is defined in the motivic
stable homotopy category in terms of maps from the mapping cone
of $D\rightarrow X$ into the Eilenberg-MacLane $H\Lambda$
spectrum representing motivic cohomology,
see \ref{sec:RMC**}, \ref{prop:DMrep}, \ref{prop:rel-mot-coh-equid}.

\begin{thm}\label{thm:MT-1}
Let $X$ be a separated scheme of finite type over $k$ and let 
$D \subset X$ be a closed subscheme. Then the following hold.
\begin{enumerate}
\item
There
exists a strongly convergent spectral sequence
\[
E^{a,b}_2 = H^{a-b}(X|D, \Lambda(-b))
\Rightarrow KH_{-a-b}(X,D)_{\Lambda}.
\]
This spectral sequence degenerates with rational coefficients.
\item
The spectral sequence exists with integral coefficients if $X$ 
and $Y$ are regular.
\item
If $X$ is regular and $D \subset X$ is a Cartier divisor, then
$H^{a}(X|D, \Lambda(b)) \cong \H^{a-2b}_{cdh}(X, \Lambda_{X|D}(b)[2b])$.
\item
If $X$ is projective over $k$, then 
$H^{a}(X|D, \Lambda(b)) = H^a_c(X\setminus D, \Lambda(b))$ is the 
Friedlander-Voevodsky
motivic cohomology with compact support of $X \setminus D$.
\end{enumerate}    
\end{thm}

\begin{thm}\label{thm:MT-2}
If $X$ has dimension $d$, then:
\begin{enumerate}
\item
There exists a cycle class map
\[
cyc_i: H^{2d+i}(X|D, \Lambda(d+i)) \to KH_{i}(X,D)_{\Lambda}.
\]
\item
If $k$ admits resolution of singularities, $X$ is regular and $D \subset X$ is 
an effective Cartier divisor, there exist Chern class maps
\[
c_{X|D, a,b}: KH_a(X,D)_{\Lambda} \to H^{2b-a}(X|D, \Lambda(q))
\]
which are functorial in the pair $(X,D)$.
\end{enumerate}
\end{thm}

\begin{thm}\label{thm:MT-3}
If $X$ is regular, $D \subset X$  an effective Cartier divisor and
$i \ge 0$ an integer, there exists a homomorphism 
\[
\lambda_{X|D}: \CH^i(X|D)_{\Lambda} \to H^{2i}(X|D, \Lambda(i)).
\]

If $k$ is furthermore algebraically closed,
$X$ is affine of dimension $d$ and 
$D$ is regular, then there is an isomorphism
$\lambda_{X|D}: \CH^d(X|D) \xrightarrow{\cong} H^{2d}(X|D, \Z(d))$.
\end{thm}

\subsection{Outline of proofs}\label{sec:Outline}
We end this section with a brief outline of our proofs.
The idea of the construction of the spectral sequence for relative
$K$-theory came from our previous work \cite{KP}, where such a 
spectral sequence was constructed for the $KH$-theory of singular schemes.
Extending our techniques, we feed the machinery of the slice filtration into
the relative setting. By mapping the mapping cone of a closed immersion
of schemes into the slice filtration of $\KGL$ and generalizing some
results of \cite{KP} to the relative case, we obtain the desired 
spectral sequence and its rational degeneration. 

In order to get
a tower for the relative $K$-theory spectrum leading to the spectral
sequence, we need to
use \cite{Levine-3} which compares Voevodsky's slice filtration 
with Levine's homotopy coniveau tower. This yields a tower for relative 
$K$-theory whose layers are identified with the relative motivic 
cohomology.

The remaining part of this text is devoted to showing a direct 
relation between the $E_2$-terms of the spectral sequence with the Chow groups 
with modulus. In \S~\ref{sec:Chow-mot-coh}, we construct a homomorphism from 
the higher Chow groups with
modulus to the $E_2$-terms by again using the comparison
between the slice and the homotopy coniveau tower for $K$-theory.
We show in the final section that these maps are isomorphisms 
in the 0-cycle range for affine schemes. 
This critically uses the affine Roitman torsion theorem of \cite{Krishna-2} as
the main input.

\section{Review of motivic spaces and 
algebraic cycles}\label{sec:review}
In this section, we fix our notations and provide a limited recollection of
some definitions and known results related to the stable homotopy category
of smooth schemes over a base scheme. We recall the definitions of 
cycles with modulus on smooth schemes and Levine-Weibel Chow groups of 
singular schemes. This Chow group of singular schemes will play a crucial
role in our comparison between Chow group of 0-cycles with modulus and
relative motivic cohomology.

\subsection{Definitions and Notation}\label{sec:Notn}
We will write $k$ for a perfect field of exponential characteristic $p$ (in
some cases we will assume that the field $k$ admits resolution of 
singularities 
\cite[Definition~3.4]{FV}).  
Let $\Sch_k$ be the category of separated schemes of 
finite type over $k$ and $\Sm_k$ be the full subcategory of $\Sch_k$ 
consisting of
smooth schemes over $k$.  If $X\in \Sch_k$, we will write $\Sm_X$ for the full
subcategory of $\Sch_k$ consisting of smooth schemes over $X$.  
Let $(\Sm_k)_{Nis}$ (resp. $(\Sm_X)_{Nis}$, 
$(\Sch_k)_{cdh}$, $(\Sch_k)_{Nis}$) denote $\Sm_k$ equipped with the Nisnevich
topology (resp. $\Sm_X$ equipped with the Nisnevich topology, $\Sch_k$ equipped
with the $cdh$-topology, $\Sch _k$ equipped with the Nisnevich topology).  
To simplify
the notation we will write $X \times Y$ for $X \times _{\Spec(k)} Y$.

Let $\unsmot$ (resp. $\unsmot _{X}$, $\unsmotcdh$) be the category of pointed 
simplicial presheaves on $\Sm_k$ (resp. $\Sm_X$, $\Sch_k$) equipped with the  
motivic model structure described in \cite{Isak}
considering the Nisnevich topology on $\Sm_k$ (resp. Nisnevich topology on
$\Sm_X$, $cdh$-topology on $\Sch_k$) and the affine line $\A^1_{k}$ as an 
interval. A simplicial presheaf will often be called a {\sl motivic space}.

Let $T$ in $\unsmot$ (resp. $\unsmot _{X}$, $\unsmotcdh$) be the pointed 
simplicial presheaf
represented by $S^{1}_s \wedge S^1_t$, where $S^1_t$ is 
$\A^1_{k} \setminus  \{ 0 \}$ (resp. $\A^1_{X} \setminus \{ 0\}$, 
$\A^1_{k} \setminus \{ 0\}$) pointed by $1$, and $S^{1}_s$ denotes the 
simplicial circle.  
Given an arbitrary integer $r \geq1$, let $S^{r}_s$
(resp. $S^{r}_t$) denote the iterated smash product of $S^{1}_s$ 
(resp. $S^1_t$) 
with $r$-factors: $S^{1}_s\wedge \cdots \wedge S^{1}_s$ 
(resp. $S^1_t \wedge \cdots
\wedge S^1_t$); $S^{0}_{s} = S^{0}_{t}$ will be by definition equal to the 
pointed simplicial presheaf represented by the base scheme $\Spec(k)$ 
(resp. $X$, $\Spec(k)$). 
 
Let $\Tspectra$ (resp. $\TspectraX$, $\Tspectracdh$) denote the category of 
symmetric $T$-spectra on $\unsmot$ (resp. $\unsmot _{X}$, $\unsmotcdh$) 
equipped with 
the motivic model structure defined in \cite[8.7]{Hovey}.
We will write $\stablehomotopy$ (resp. 
$\stablehomotopy _{X}$,
$\stablehomotopycdh$) for the homotopy category of $\Tspectra$ (resp. 
$\TspectraX$, $\Tspectracdh$)
which is a tensor triangulated category.
 For any two integers $m$, $n \in \Z$, let $\Sigma^{m,n}$ denote the
automorphism $\Sigma^{m-n}_s \circ \Sigma^n_t: \stablehomotopy \to
\stablehomotopy$ (this also makes sense in $\stablehomotopy _{X}$ and 
$\stablehomotopycdh$).
We will write $\Sigma _{T}^{n}$ for $\Sigma ^{2n,n}$, and
$E\wedge F$ for the smash product of $E$, $F\in \stablehomotopy$ (resp. 
$\stablehomotopy _{X}$, $\stablehomotopycdh$).

Given a simplicial presheaf $A$, we will write $A_{+}$ for the pointed 
simplicial presheaf
obtained by adding a disjoint base point (isomorphic to the base scheme) to 
$A$.  For any $B \in \unsmot$, let $\Sigma^{\infty}_T(B)$ denote the
object $(B, T \wedge B, \cdots ) \in \Tspectra$.
This functor makes sense for objects in $\unsmotcdh$ and $\unsmot _{X}$
as well.  

If $F:\mathcal A\rightarrow \mathcal B$ is a
functor with right adjoint $G:\mathcal B \rightarrow \mathcal A$, 
we shall say that $(F,G):
\mathcal A \rightarrow \mathcal B$ is an adjunction.  
We will use the following notation in all the categories under 
consideration: $\ast$ 
will denote the terminal object, and $\cong$ will denote that a map 
(resp. functor) is an isomorphism (resp. equivalence of categories).

Throughout this paper, $\Lambda$ will denote the ring $\Z$ if $k$ admits 
resolution of singularities, or the ring $\Z[\tfrac{1}{p}]$ otherwise.
For any abelian group $M$, we shall let 
$M_{\Lambda} = M {\underset{\Z}\otimes} \Lambda$.

\enlargethispage{60pt}
 
\subsection{Some known results in motivic homotopy theory}
\label{ssec:preres}
Let $X\in \Sch_k$ and let $v: X \to \Spec(k)$ denote the structure map.
Let $\pi: (\Sch_k)_{cdh}\rightarrow (\Sm_k)_{Nis}$ be the canonical
continuous map of sites. 
We will write $(\pi ^{\ast}, \pi _{\ast}):\unsmot \rightarrow 
\unsmotcdh$,
$(v ^{\ast}, v _{\ast}):\unsmot \rightarrow \unsmot _{X}$ for the adjunctions 
induced by $\pi$, $v$ respectively. 
We will also consider the morphism of sites 
$\pi _{X}: (\Sch_k)_{cdh}\rightarrow (\Sm_X)_{Nis}$
and the corresponding adjunction $(\pi_{X} ^{\ast}, \pi_{X \ast}):\unsmot _{X} 
\rightarrow \unsmotcdh$. The following result can be found
in \cite[2.4 and 2.9]{KP}.

\begin{prop}\label{prop:commdiag1}
The adjunctions $(\pi ^{\ast}, \pi _{\ast}):\unsmot \rightarrow \unsmotcdh$,
$(v ^{\ast}, v _{\ast}):\unsmot \rightarrow \unsmot _{X}$,
$(\pi_{X} ^{\ast}, \pi_{X \ast}):\unsmot _{X} 
\rightarrow \unsmotcdh$ are Quillen adjunctions.
\end{prop}

We further conclude from ~\propref{prop:commdiag1}
and \cite[Theorem~9.3]{Hovey} the following:

\begin{prop}\label{prop:tablebasech}
The pairs $(\pi ^{\ast}, \pi _{\ast})$, $(v ^{\ast}, v _{\ast})$ and
$(\pi _{X}^{\ast}, \pi _{X\ast})$ are Quillen adjunctions between stable model 
categories.
\end{prop}

We deduce from ~\propref{prop:tablebasech} that there are pairs of
adjoint functors $(\mathbf L \pi ^{\ast}, \mathbf R \pi _{\ast}):\stablehomotopy 
\rightarrow \stablehomotopycdh$,
$(\mathbf L v^{\ast}, \mathbf R v _{\ast}):\stablehomotopy \rightarrow 
\stablehomotopy _{X}$ and
$(\mathbf L \pi _{X}^{\ast}, \mathbf R \pi _{X\ast}):\stablehomotopy _{X} 
\rightarrow \stablehomotopycdh$ between the various stable homotopy
categories of motivic $T$-spectra.  
We observe that for $a\geq b\geq 0$, the suspension functor $\Sigma^{a,b}$
in $\stablehomotopy$ (resp. $\stablehomotopy _X$, $\stablehomotopycdh$)
is the derived functor of the left Quillen functor 
$E\mapsto S^{a-b}_s\wedge S^b_t\wedge E$
in $\Tspectra$ (resp. $\TspectraX$, $\Tspectracdh$).
Since the functors $\pi ^{\ast}$, $v^{\ast}$, $\pi _{X}^{\ast}$
are simplicial and symmetric monoidal, we deduce that they commute 
with the suspension functors $\Sigma^{m,n}$, i.e.,
for every $m$, $n\in \mathbb Z$:
$ \mathbf L \pi ^{\ast} \circ \Sigma^{m,n} (-) \cong
\Sigma^{m,n} \circ \mathbf L \pi ^{\ast} (-)$,
$\mathbf L v ^{\ast} \circ \Sigma^{m,n} (-) \cong
\Sigma^{m,n} \circ \mathbf L v ^{\ast} (-)$ and
$ \mathbf L \pi ^{\ast}_{X} \circ \Sigma^{m,n} (-) \cong
\Sigma^{m,n} \circ \mathbf L \pi ^{\ast}_{X} (-)$.

\subsection{Higher Chow groups with modulus}\label{sec:HCGM}
For $n \ge 1$, let $\square^n$ denote the scheme 
$\A^n_k \cong (\P^1_k \setminus \{\infty\})^n$.
Let $(y_1, \cdots , y_n)$ denote the coordinate of a point on $\square^n$.
We shall denote the scheme $(\P^1_k)^n$ by $\ov{\square}^n$. 
For $1 \le i \le n$, let
$F^{\infty}_{n,i}$ denote the closed subscheme of $\ov{\square}^n$ given by the 
equation $\{y_i = \infty\}$. We shall denote the divisor 
$\stackrel{n}{\underset{i =1}\sum} F^\infty_{n,i}$ by $F^\infty_n$.

Let $X$ be a smooth quasi-projective scheme of dimension $d \ge 0$ over 
$k$ and let $D \subset X$ be an effective Cartier divisor.
For $r \in \Z$ and $n \ge 0$, let $\un{z}_r(X|D,n)$ be the free abelian group 
on integral closed subschemes $V$ of $X \times \square^n$ of dimension 
$r+n$ satisfying the following conditions.

\noindent $(1)$ (Face condition) For each face $F$ of $\square^n$, 
$V$ intersects  $X \times F$ properly:
\[
\dim_k(V\cap (X\times F)) \le r+\dim_k (F), {\rm and} 
\] 

\noindent $(2)$ (Modulus condition) $V$ is a cycle with modulus $D$ relative
to $F^{\infty}_n$:
\[
\nu^*(D \times \ov{\square}^n) \le \nu^*(X \times F^\infty_n),
\]
where $\ov{V}$ is the closure of $V$ in $X \times \ov{\square}^n$
and $\nu: \ov{V}^N \to \ov{V} \to X \times \ov{\square}^n$ is the
composite map from the normalization of $\ov{V}$. 
We let ${\un{z}_r(X|D, n)_{\dgn}}$ denote the subgroup of
$\un{z}_r(X|D, n)$ generated by cycles which are pull-back of
some cycles under various projections $X \times \square^n \to X \times
\square^{m}$ with $m < n$.

\begin{defn}
The {\em cycle complex with modulus} $(z_r(X|D, \bullet), d)$ of $X$ in 
dimension $r$ and with modulus $D$ is the non-degenerate complex associated to 
the cubical abelian group $\un{n} \mapsto \un{z}_{r}(X|D, n)$, i.e.,
\[
z_r(X|D, n): = \frac{\un{z}_r(X|D, n)}
{\un{z}_r(X|D, n)_{\dgn}}.
\]

The homology $\CH_r(X|D, n): = H_n (z_r(X|D, \bullet))$
is called a \emph{higher Chow group} of $X$ with modulus $D$.
Sometimes, we also write it as the Chow group of the {\sl modulus pair} $(X,D)$.
If $X$ has dimension $d$, we write 
$\CH^r(X|D, n) = \CH_{d-r}(X|D, n)$. We shall often write $\CH^r(X|D,0)$
as $\CH^r(X|D)$. We refer to \cite{KPark} for further details on this
definition. The reader should note that $\CH_r(X|D, n)$ coincides with
the usual higher Chow group of Bloch $\CH_r(X, n)$ if $D = \emptyset$.
\end{defn}

\subsection{Levine-Weibel Chow group of singular schemes}
\label{sec:Chow-sing}
We recall the definition of the cohomological Chow group of 0-cycles for
singular schemes from \cite{BK} and \cite{LW}. 
Let $X$ be a reduced quasi-projective
scheme of dimension $d \ge 1$ over $k$. Let $X_{\rm sing}$ and $X_{\rm reg}$
respectively denote the loci of the singular and the regular points of $X$.
We let $X^N$ denote the normalization of $X$.
Given a nowhere dense closed subscheme $Y \subset X$ such that 
$X_{\rm sing} \subseteq Y$
and no component of $X$ is contained in $Y$, we let $\sZ_0(X,Y)$ denote
the free abelian group on the closed points of $X \setminus Y$.
We write $\sZ_0(X, X_{\rm sing})$ in short as $\sZ_0(X)$.

\begin{defn}\label{defn:0-cycle-S-1}
Let $C$ be a pure dimension one reduced scheme in $\Sch_k$. 
We shall say that a pair $(C, Z)$ is \emph{a good curve
relative to $X$} if there exists a finite morphism $\nu\colon C \to X$
and a  closed proper subscheme $Z \subsetneq C$ such that the following hold.
\begin{enumerate}
\item
No component of $C$ is contained in $Z$.
\item
$\nu^{-1}(X_{\rm sing}) \cup C_{\rm sing}\subseteq Z$.
\item
$\nu$ is local complete intersection at every 
point $x \in C$ such that $\nu(x) \in X_{\rm sing}$. 
\end{enumerate}
\end{defn}

Let $(C, Z)$ be a good curve relative to $X$ and let 
$\{\eta_1, \cdots , \eta_r\}$ be the set of generic points of $C$. 
Let $\sO_{C,Z}$ denote the semilocal ring of $C$ at 
$S = Z \cup \{\eta_1, \cdots , \eta_r\}$.
Let $k(C)$ denote the ring of total
quotients of $C$ and write $\sO_{C,Z}^\times$ for the group of units in 
$\sO_{C,Z}$. Notice that $\sO_{C,Z}$ coincides with $k(C)$ 
if $|Z| = \emptyset$. 
As $C$ is Cohen-Macaulay, $\sO_{C,Z}^\times$  is the subgroup of $k(C)^\times$ 
consisting of those $f$ which are regular and invertible in the local rings 
$\sO_{C,x}$ for every $x\in Z$. 

Given any $f \in \sO^{\times}_{C, Z} \inj k(C)^{\times}$, we denote by  
${\rm div}_C(f)$ (or ${\rm div}(f)$ in short) 
the divisor of zeros and poles of $f$ on $C$, which is defined as follows. If 
$C_1,\ldots, C_r$ are the irreducible components of $C$, 
and $f_i$ is the factor of $f$ in $k(C_i)$, we set 
${\rm div}(f)$ to be the $0$-cycle $\sum_{i=1}^r {\rm div}(f_i)$, where 
${\rm div}(f_i)$ is the usual 
divisor of a rational function on an integral curve in the classical sense. 
As $f$ is an invertible 
regular function on $C$ along $Z$, ${\rm div}(f)\in \sZ_0(C,Z)$.

By definition, given any good curve $(C,Z)$ relative to $X$, we have a 
push-forward map $\sZ_0(C,Z)\xrightarrow{\nu_{*}} \sZ_0(X)$.
We shall write $\sR_0(C, Z, X)$ for the subgroup
of $\sZ_0(X)$ generated by the set 
$\{\nu_*({\rm div}(f))| f \in \sO^{\times}_{C, Z}\}$. 
Let $\sR_0(X)$ denote the subgroup of $\sZ_0(X)$ generated by 
the image of the map $\sR_0(C, Z, X) \to \sZ_0(X)$, where
$(C, Z)$ runs through all good curves relative to $X$.
We let $\CH_0(X) = \frac{\sZ_0(X)}{\sR_0(X)}$.

If we let $\sR^{LW}_0(X)$ denote the subgroup of $\sZ_0(X)$ generated
by the divisors of rational functions on good curves as above, where
we further assume that the map $\nu: C \to X$ is a closed immersion,
then the resulting quotient group ${\sZ_0(X)}/{\sR^{LW}_0(X)}$ is
denoted by $\CH^{LW}_0(X)$. Such curves on $X$ are called the 
{\sl Cartier curves}. There is a canonical surjection
$\CH^{LW}_0(X) \surj \CH_0(X)$. The Chow group $\CH^{LW}_0(X)$ was
discovered by Levine and Weibel \cite{LW} in an attempt to describe the
Grothendieck group of a singular scheme in terms of algebraic cycles.
The modified version $\CH_0(X)$ was introduced in \cite{BK}.


\subsection{The double and its Chow group}
\label{sec:double}
Let $X$ be a smooth quasi-projective scheme of dimension $d$ over $k$
and let $D \subset X$ be an effective Cartier divisor. Recall from 
\cite[\S~2.1]{BK} that the double of $X$ along $D$ is a quasi-projective
scheme $S(X,D) = X \amalg_D X$ so that
\begin{equation}\label{eqn:rel-et-2}
\begin{array}{c}
\xymatrix@R=1pc{
D \ar[r]^-{\iota} \ar[d]_{\iota} & X \ar[d]^{\iota_+} \\
X \ar[r]_-{\iota_-} & S(X,D)}
\end{array}
\end{equation}
is a co-Cartesian square in $\Sch_k$.
In particular, the identity map of $X$ induces a finite map
$\nabla:S(X,D) \to X$ such that $\nabla \circ \iota_\pm = {\rm Id}_X$
and $\pi = \iota_+ \amalg \iota_-: X \amalg X \to S(X,D)$ 
is the normalization map.
We let $X_\pm = \iota_\pm(X) \subset S(X,D)$ denote the two irreducible
components of $S(X,D)$.  We shall often write $S(X,D)$ as $S_X$
when the divisor $D$ is understood. $S_X$ is a reduced quasi-projective
scheme whose singular locus is $D_{\rm red} \subset S_X$. 
It is projective whenever
$X$ is so. It follows from \cite[Lemma~2.2]{Krishna-3} that  
~\eqref{eqn:rel-et-2}
is also a Cartesian square. 

It is clear that the map
$\sZ_0(S_X,D)\xrightarrow{(\iota^*_{+}, \iota^*_{-})} \sZ_0(X_+,D) \oplus 
\sZ_0(X_-,D)$ is an isomorphism. 
Notice also that there are push-forward inclusion maps
${p_{\pm}}_* \colon \sZ_0(X,D) \to \sZ_0(S_X,D)$ such that 
$\iota^*_{+} \circ {p_{+}}_* = {\rm Id}$
and $\iota^*_{+} \circ {p_{-}}_* = 0$.
The fundamental result that connects the 0-cycles with modulus on $X$ and
0-cycles on $S_X$ is the following.

\begin{thm}$($\cite[Theorem~1.12]{BK}$)$\label{thm:BS-main}
Let $X$ be a smooth quasi-projective scheme over $k$
and let $D \subset X$ be an effective Cartier divisor. Then there
is a split short exact sequence
\[
0 \to \CH_0(X|D) \xrightarrow{ {p_{+}}_*} \CH_0(S_X) \xrightarrow{\iota^*_-}
\CH_0(X) \to 0.
\]
\end{thm}

\vskip .2cm

\section{Relative Homotopy Invariant $K$-theory}\label{sec:relKH}
Let $Spt _{S^1}$ be the category of Bousfield-Friedlander $S^1$-spectra 
equipped with the stable model structure \cite[\S~1.2]{Levine-3}, and
$\mathcal H (S^1)$ its homotopy category.
We will write $Spt_{S^1}(\Sch _k)$ for the category of presheaves 
on $\Sch _k$ with values in $Spt _{S^1}$. We let $K \in Spt_{S^1}(\Sch _k)$
denote the Thomason-Trobaugh algebraic $K^B$-theory spectrum, and let
$KH \in Spt_{S^1}(\Sch _k)$ denote Weibel's homotopy invariant
$K$-theory spectrum \cite{Weibel}, \cite[2.8]{Cisinski}.  Let
$K_n(X)$, $n\in \mathbb Z$ be the $n$-th homotopy group of
the spectrum $K(X)$. We use similar notation for the homotopy groups of
$KH(X)$. 

\begin{defn}  \label{def.relKH}
Let $f:Y\rightarrow X$ be a map in $\Sch _k$.
\begin{enumerate}
\item  \label{def.relKH.a} We will write $K(f)$ (resp. $KH(f)$)
for the homotopy fiber
of the map $f^\ast:K(X)\rightarrow K(Y)$ (resp.
$f^\ast:KH(X)\rightarrow KH(Y)$) induced by $f$ in $Spt_{S^1}$.
If $f$ is a closed (resp. open) immersion, we will write $K(X,Y)$ (resp. 
$K^{X\backslash Y}(X)$) for $K(f)$. Analogous notations will be used for
$KH(f)$ too. We will write $K_n(f)$ (resp. $KH_n(f)$)
for the $n$-th homotopy group of the spectrum $K(f)$ (resp. $KH(f)$).

\item \label{def.relKH.b} We will write $M_f$ for the mapping cone of $f$ in 
$\unsmotcdh$.  Namely, let $X_+$, $Y_+$
denote the simplicial presheaves represented by $X$, $Y$ with
a disjoint base point.  We then factor $f$ in $\unsmotcdh$:
\[ \xymatrix@R=1pc{Y_+ \ar[r]^-f \ar[dr]_-{c_f}& X_+\\
	&A_f \ar[u]_-{w_f}}
\]
where $c_f$ (resp. $w_f$) is a cofibration (resp. trivial fibration) in 
$\unsmotcdh$ 
and let $M_f$ be the pushout in $\unsmotcdh$ of the diagram: 
$\ast \leftarrow Y_+ \stackrel{c_f}{\rightarrow} A_f$.
If $f$ is a closed immersion, $M_f$ is canonically 
identified with the quotient $X/Y$ in $\unsmotcdh$.
\end{enumerate}
\end{defn}

Since the cofibrant replacement functor is functorial in $\unsmotcdh$, 
we deduce that every commutative diagram in $\Sch _k$:
\begin{align}  \label{dg.map.pairs}
\begin{array}{c}
\xymatrix@R=1pc{Y' \ar[r]^-{f'} \ar[d]_-{g'}& X' \ar[d]^-{g}\\
	Y \ar[r]^-{f}&X}
\end{array}
\end{align}
induces a commutative diagram in $\unsmotcdh$ which is natural in 
$g$ and $g'$ \eqref{dg.map.pairs}:
\begin{align}  \label{dg.map.cofs}
\begin{array}{c}
\xymatrix@R=1pc{Y'_+\ar[d]_-{g'} \ar[r]^-{c_{f'}}& A_{f'}\ar[r] 
\ar[d]^-{(g,g')_A}& 
M_{f'} \ar[d]^-{(g,g')_M}\\
	Y_+ \ar[r]^-{c_{f}}& A_{f}\ar[r] & M_{f}}
\end{array}
\end{align}
where the rows are cofiber sequences in $\unsmotcdh$.

\subsection{Voevodsky's $\KGL$ spectrum}  \label{ss.KGL}
For any Noetherian separated scheme $X$ of finite Krull dimension, the motivic 
$T$-spectrum $\KGL_X \in Spt({\sM_X})$ was defined by Voevodsky
(see \cite[\S~6.2]{Voev-0}).  It represents
algebraic $K$-theory of objects in $\Sm_X$ if $X$ is regular.
It was later shown by Cisinski \cite{Cisinski} that for $X$ not necessarily 
regular, $\KGL_X$ represents Weibel's homotopy invariant $K$-theory 
$KH_*(Y)$ for $Y \in \Sm_X$.  We will write $\KGL \in \Tspectracdh$ for
$\mathbf L \pi ^{\ast} \KGL _k$.

We will write $Spt_{S^1}(\unsmotcdh )$ for the category of symmetric 
$S^1$-spectra on $\unsmotcdh$ equipped with the 
motivic model structure defined in \cite[8.7]{Hovey},
and $\stablehomotopycdh ^{S^1}$ for its homotopy category which is a
tensor triangulated category.  
We shall denote the homotopy category of $Spt_{S^1}(\unsmot )$ by
$\sS\sH_{nis}^{S^1}$. There
is a Quillen adjunction \cite[\S 2]{Voev-newapp}:
\begin{equation}  \label{eq.s1Tadj}
(\Sigma _t ^{\infty},\Omega _t ^{\infty}):Spt_{S^1}(\unsmotcdh )\rightarrow 
\Tspectracdh.
\end{equation}

Consider the functor of global sections on $X\in \Sch _k$,
$\Gamma _X: Spt_{S^1}(\unsmotcdh ) \rightarrow Spt_{S^1}$, 
which admits a left adjoint $X_+\otimes -: Spt_{S^1}\rightarrow 
Spt_{S^1}(\unsmotcdh )$, 
$(E^0, E^1,\ldots) \mapsto (X_+\wedge E^0, X_+\wedge E^1, \ldots )$.
Since $X_+\otimes -$ is a left Quillen functor, we obtain a Quillen adjunction:
\begin{equation}  \label{eq.globalsec}
(X_+\otimes -,\Gamma _X):Spt_{S^1}\rightarrow Spt_{S^1}(\unsmotcdh ).
\end{equation}

We will write $\mathbf R\Gamma _X: 
\stablehomotopycdh ^{S^1}\rightarrow \mathcal H (S^1)$
(resp. $\mathbf R\Omega _{t}^{\infty}: \stablehomotopycdh \rightarrow 
\stablehomotopycdh ^{S^1}$) 
for the right derived functor of $\Gamma _X$ (resp. $\Omega _{t}^{\infty}$), 
$\mathbf{RHom}$ for the internal 
Hom-functor in $\stablehomotopycdh ^{S^1}$ and $\Sigma _{S^1}^{\infty}:\unsmotcdh
\rightarrow Spt_{S^1}(\unsmotcdh )$, $A\mapsto (A,S^1\wedge A, \ldots)$ for the
infinite suspension functor.  When $X=\Spec(k)$, we will simply write 
$\mathbf R\Gamma _k$.

\begin{remk} \label{rmk.loc-gl}
Since $X_+\otimes -$ is naturally isomorphic to the composition 
$\Sigma _{S^1}^{\infty}X_+\wedge(\Spec(k) _+ \otimes -)$, we deduce that
$\mathbf R\Gamma _X$ is naturally isomorphic to 
$\mathbf R\Gamma _k \circ \mathbf{RHom}(\Sigma^{\infty}_{S^1}X_+,-)$.
\end{remk}

By construction and Definition~\ref{def.relKH}\eqref{def.relKH.b}, there is 
a commutative diagram in $Spt_{S^1}(\unsmotcdh )$:
\[
\xymatrix@R=1pc{\Sigma _{S^1}^{\infty}Y_+ \ar[r] \ar[d]_-f& 
\Sigma _{S^1}^{\infty}A_f \ar[r]
	\ar[dl]^-{w_f}& \Sigma _{S^1}^{\infty}M_f \\
	\Sigma _{S^1}^{\infty}X_+, &&} 
\]
where the top row is a cofiber sequence and $w_f$ is a $S^1$-stable weak
equivalence. So this
induces a commutative diagram where the solid arrows form a fiber sequence in 
$\stablehomotopycdh ^{S^1}$ (recall that $\KGL$ is by definition 
$\mathbf L \pi ^{\ast} \KGL _k$) and $w^*_f$ is an isomorphism
(see the last paragraph of \cite[Page~592]{Voev-0}):
\[
\xymatrix@R=1pc{
\mathbf{RHom}(\Sigma _{S^1}^{\infty}M_f, \mathbf R\Omega _{t}^{\infty}\KGL) \ar[r] 
& 
\mathbf{RHom}(\Sigma _{S^1}^{\infty}A_f, \mathbf R\Omega _{t}^{\infty}\KGL) 
\ar[d]\\
\mathbf{RHom}(\Sigma _{S^1}^{\infty}X_+, \mathbf R\Omega _{t}^{\infty}\KGL) 
\ar@{-->}[r]_-{f^\ast} \ar@{-->}[ur]^(.35){w_f^{\ast}}
& \mathbf{RHom}(\Sigma _{S^1}^{\infty}Y_+, \mathbf R\Omega _{t}^{\infty}\KGL).}
\]
Thus, applying $\mathbf R\Gamma _k$, we obtain the commutative diagram 

\begin{align}  \label{dg.distreltr}
\begin{array}{c}
\xymatrix@R=1pc@C=1pc{
\mathbf R\Gamma _k \circ \mathbf{RHom}(\Sigma _{S^1}^{\infty}M_f, 
\mathbf R\Omega _{t}^{\infty}\KGL)\ar[r] &
\mathbf R\Gamma _k \circ \mathbf{RHom}(\Sigma _{S^1}^{\infty}A_f, 
\mathbf R\Omega _{t}^{\infty}\KGL) \ar[d] \\
\mathbf R\Gamma _k \circ \mathbf{RHom}(\Sigma _{S^1}^{\infty}X_+, 
\mathbf R\Omega _{t}^{\infty}\KGL) 
\ar@{-->}[r]_-{f^\ast} \ar@{-->}[ur]^(.35){w_f^{\ast}}& 
\mathbf R\Gamma _k \circ \mathbf{RHom}(\Sigma _{S^1}^{\infty}Y_+, 
\mathbf R\Omega _{t}^{\infty}\KGL).}
\end{array}
\end{align}
where the 
solid arrows form a fiber sequence and $w^*_f$ is
an isomorphism in $\mathcal H (S^1)$.

\begin{lem}\label{lm.secpres}
The composition of $w_f^\ast$ with the right vertical arrow 
in \eqref{dg.distreltr} is canonically identified with 
the pull-back map $f^{\ast}:KH(X)\rightarrow KH(Y)$ in  $\mathcal H(S^1)$.
\end{lem}
\begin{proof}
The lemma follows immediately by combining \S 2.16 (see p.~438, line~9)
and Proposition~2.19 of \cite{Cisinski}, which together imply that there is a
canonical isomorphism $\mathbf R\Omega _{t}^{\infty}\KGL \cong KH$ in
$\stablehomotopycdh ^{S^1}$.  
\end{proof}

\begin{cor}\label{cor.hfibident}
$\mathbf R\Gamma _k \circ \mathbf{RHom} (\Sigma _{S^1}^{\infty}M_f, 
\mathbf R\Omega _{t}^{\infty}\KGL)
\cong KH(f)$ in $\mathcal H (S^1)$ and therefore
$KH_n(f)\cong \Hom _{\stablehomotopycdh}(\Sigma _T ^{\infty}M_f[n], \KGL)$.
In particular, for a closed immersion $f:Y \inj X$, we have
\begin{equation}\label{eqn:CI}
KH_n(X,Y) \cong 
\Hom _{\stablehomotopycdh}(\Sigma _T ^{\infty}(X/Y)_+[n], 
\KGL).
\end{equation}
\end{cor}
\begin{proof}
The corollary follows directly from \eqref{dg.distreltr}, \lemref{lm.secpres}, 
Quillen adjunctions \eqref{eq.s1Tadj} and \eqref{eq.globalsec},  
representability result \cite[Th{\'e}or{\`e}me~2.20]{Cisinski},
and Definition~\ref{def.relKH}\eqref{def.relKH.a}.
\end{proof}

In this paper, we shall write $KH(f)$ and
$\mathbf R \Gamma _k \circ \mathbf{RHom} 
(\Sigma _{S^1}^{\infty}M_f, \mathbf R\Omega _{t}^{\infty}\KGL)$
interchangeably hereafter.
It follows from \eqref{dg.distreltr},
\lemref{lm.secpres} and \corref{cor.hfibident}
that the commutative diagram \eqref{dg.map.cofs}, associated to
\eqref{dg.map.pairs},
induces $(g,g')^{\ast}:KH(f)\rightarrow KH(f')$ in $\mathcal H (S^1)$ which 
fits in a morphism of distinguished triangles in $\mathcal H (S^1)$:
\begin{align}\label{dg.natrelKH}
\begin{array}{c}
\xymatrix@R=1pc{KH(f)\ar[d]_-{(g,g')^{\ast}} \ar[r]& KH(X)\ar[r]^-{f^\ast} 
\ar[d]^-{g^{\ast}}& 
KH(Y) \ar[d]^-{g'^{\ast}}\\
	KH(f') \ar[r]& KH(X')\ar[r]^-{f'^\ast} & KH(Y').}
\end{array}
\end{align}

\begin{remk}\label{rmk.natrelKH}
By construction, the diagram \eqref{dg.natrelKH} is natural in $(g,g')$
in \eqref{dg.map.pairs}.
\end{remk}

\section{Relative motivic cohomology}\label{sec:MCS}
We continue to assume that $k$ is a perfect field of exponential 
characteristic $p$. In this section, we define our relative motivic
cohomology for a closed immersion of schemes $Y \subset X$ in $\Sch_k$.
We shall then show that this relative motivic cohomology can be
described as the $cdh$-hypercohomology of a presheaf of complexes of
equidimensional cycles. These relative cohomology groups will 
later constitute the $E_2$-terms of
our spectral sequence for relative $KH$-theory.

\subsection{Motivic cohomology of singular schemes}\label{sec:Mot-sing}
Recall from \cite[Lecture~16]{MVW} that given $T \in \Sch_k$ and an integer
$r \ge 0$, the presheaf $z_{equi}(T,r)$ on $\Sm_k$ is defined by
letting $z_{equi}(T,r)(U)$ be the free abelian group generated by the closed
and irreducible subschemes $Z \subsetneq U \times T$ which are dominant and
equidimensional of relative dimension $r$ (any fiber is either empty or all
its components have dimension $r$) over a component of $U$. 
It is known that $z_{equi}(T,r)$ is a sheaf on the big 
{\'e}tale site of $\Sm_k$.

Let $C_{\ast}z_{equi}(T,r)$ denote the chain complex of
presheaves of abelian groups associated, via the Dold-Kan correspondence, to
the simplicial presheaf on $\Sm_k$ given by
$C_nz_{equi}(T,r)(U) = z_{equi}(T,r)(U \times \Delta^n_k)$.
The simplicial structure on $C_{\ast}z_{equi}(T,r)$ 
is induced by the cosimplicial scheme  $\Delta^{\bullet}_k$. 
Recall the following definition of motivic cohomology of singular schemes
from \cite[Definition~9.2]{FV}.

\begin{defn}\label{defn:MC-sing}
The motivic cohomology groups of $X \in \Sch_k$ are defined as the
hypercohomology 
\[
\begin{array}{lll}
H^m(X, \Z(n)) & = & 
\H^{m-2n}_{cdh}(X,  \mathbf L \pi ^{\ast}(C_{\ast}z_{equi}(\A^n_k,0)))
\\
& \cong & \H^{m-2n}_{cdh}(X,  C_{\ast}z_{equi}(\A^n_k,0)_{cdh}).
\end{array}
\]
\end{defn}

We will also need to consider $\mathbb Z [\tfrac{1}{p}]$-coefficients.
In this case, we will write:
\[  H^m(X, \Z [\tfrac{1}{p}](n))=
\H^{m-2n}_{cdh}(X,  \mathbf L \pi ^{\ast}(C_{\ast}z_{equi}(\A^n_k,0)[\tfrac{1}{p}])).
\]

For $n < 0$, we set $H^m(X, \Z(n)) =  H^m(X, \Z [\tfrac{1}{p}](n)) = 0$.

\subsection{Motivic cohomology via $\stablehomotopycdh$}
\label{sec:Rep-MCoh}
In order to represent the motivic cohomology of a singular scheme $X$ in 
$\stablehomotopycdh$, let us recall the Eilenberg-MacLane spectrum 
\[
H\Z = (K(0,0), K(1, 2), \cdots , K(n,2n), \cdots )
\]
in $\Tspectra$, where
$K(n,2n)$ is the presheaf of simplicial abelian groups on $\Sm_k$ associated to
the presheaf of chain complexes 
$C_{\ast}(\frac{z_{equi}(\P^n_k,0)}{z_{equi}(\P^{n-1}_k,0)})$ via the Dold-Kan
correspondence. The assembly maps of this spectrum are induced by the 
canonical map
$g: \P^1_k \to C_{\ast}z_{equi}(\P^1_k, 0)$. This map assigns to any map
$U \to \P^1_k$ its graph in $U \times \P^1_k$.
This in turn descends to maps $T \wedge K(n,2n) 
\xrightarrow{\ov{g} \wedge {\rm Id}}
K(1,2) \wedge K(n,2n) \xrightarrow{\times} K(n+1, 2n+2)$, 
where the latter is the obvious external product map.
Using the localization theorem, it follows that
$K(n,2n)$ is weak equivalent to $C_{\ast}z_{equi}(\A^n_k,0)$.
We shall not distinguish between a simplicial abelian group and the associated
chain complex of abelian groups from now on in this text and will use them
interchangeably.  

By \cite[Theorem~3.10]{KP} motivic cohomology can be
defined via $\stablehomotopycdh$, so this leads naturally to the following
definition of relative motivic cohomology.
Recall our notation that $\Lambda$ denotes the ring $\Z$ if $k$ admits 
resolution of singularities, or the ring $\Z[\tfrac{1}{p}]$ otherwise.

\subsubsection{Relative motivic cohomology}\label{sec:RMC**}
Let $f:Y \to X$ be any morphism in $\Sch_k$. 
For any commutative ring $R$, we define the relative motivic cohomology of 
the pair $(X,Y)$ with coefficients in $R$ by
\begin{equation}\label{eqn:RMC-1-0}
H^m(X|Y, R(n))  = 
\Hom_{\stablehomotopycdh}(\Sigma^{\infty}_T M_f, 
\Sigma^{m,n} \mathbf L \pi^* HR).
\end{equation}
For $n < 0$, we set $H^m(X|Y, R(n)) = 0$.
Notice that by \cite[3.10]{KP}
this definition reduces to ~\ref{defn:MC-sing} when $Y = \emptyset$ and
$R = \Lambda$.

\begin{prop}\label{thm:cdhdescent}
Let $k$ be a perfect field of exponential characteristic $p$.
For any closed immersion $f: Y \inj X$ in $\Sm_k$ and integers 
$m$, $n \in \Z$, there is a natural isomorphism
\begin{equation}\label{eqn:cdhdescent-01}
\theta_{X|Y}: H^m(X|Y, \Z(n)) \xrightarrow{\cong} 
\Hom_{\stablehomotopy}(\Sigma^{\infty}_T M_f, \Sigma^{m,n} H\Z).
\end{equation}
\end{prop}
\begin{proof}
Follows from the definition of relative motivic cohomology \eqref{eqn:RMC-1-0} 
and \cite[Corollary~2.17]{KP}.
\end{proof}

\subsection{Motivic cohomology via $\DM_k$}
\label{sec:Rep-MCoh-DM}
Let ${\rm MS}^{\rm tr}$, ${\rm MS}^{\rm tr}_{cdh}$ denote the category of 
presheaves with transfers on $\Sm_k$, $\Sch _k$, respectively; equipped with
the model structure induced by the adjunction $(tr,U):\unsmot \rightarrow
{\rm MS}^{\rm tr}$, $(tr,U):\unsmotcdh \rightarrow {\rm MS}^{\rm tr}_{cdh}$,
where $U$ is the functor that forgets transfers, i.e., a map $f$ in 
${\rm MS}^{\rm tr}$, 
${\rm MS}^{\rm tr}_{cdh}$ is a weak equivalence (resp. a fibration) 
if and only if
$U(f)$ is a weak equivalence (resp. a fibration) in $\unsmot$ and $\unsmotcdh$.

We will write $Spt({\rm MS}^{\rm tr})$, $Spt({\rm MS}^{\rm tr}_{cdh})$ for the 
category of symmetric $T$-spectra on ${\rm MS}^{\rm tr}$, 
${\rm MS}^{\rm tr}_{cdh}$,
respectively; equipped with the stable model structure \cite[8.7]{Hovey}, where
$T$ is identified with $C_*(\frac{z_{equi}(\P^1_X,0)}{z_{equi}(X,0)})$.
Let ${\rm DM}_k$ (resp. ${\rm DM_{cdh}}$) denote
the homotopy category of $Spt({\rm MS}^{\rm tr})$, 
$Spt({\rm MS}^{\rm tr}_{cdh})$.  Notice that
${\rm DM}_k$ is Voevodsky's triangulated category of motives.

There is a $T$-suspension functor $\Sigma^{\infty}_T: {\rm MS}^{\rm tr} \to
Spt({\rm MS}^{\rm tr})$, $\Sigma^{\infty}_T: {\rm MS}^{\rm tr}_{cdh} \to
Spt({\rm MS}^{\rm tr}_{cdh})$, and the forgetful functor induces levelwise
Quillen adjunctions: 
\begin{align}  \label{adj.DM-SH}
(tr, U):\Tspectra \rightarrow Spt({\rm MS}^{\rm tr}), \quad
(tr, U):\Tspectracdh \rightarrow Spt({\rm MS}^{\rm tr}_{cdh}),
\end{align}
which fit in a commutative diagram of left 
Quillen functors:
\begin{align}\label{eqn:SH-DM}
\begin{array}{c}
\xymatrix@C=1pc@R=1pc{
Spt(\unsmot) \ar[r]^-{\pi ^{\ast}} \ar[d]_{tr} & Spt(\unsmotcdh)
\ar[d]^{tr} \\
Spt({\rm MS}^{\rm tr}) \ar[r]^-{\pi ^{\ast}} & Spt({\rm MS}^{\rm tr}_{cdh}).}
\end{array}
\end{align}

For $X\in \Sch _k$,  we will write $M(X)\in \DM_{cdh}$ 
for $tr(\Sigma _T ^{\infty}X)$ and $\mathbb Z(n)[m]\in \DM_{cdh}$, $m$, 
$n\in \mathbb Z$ for 
$tr(\Sigma ^{m,n}\Sigma _T ^{\infty}S^0_s)$.  
Given $E\in \DM_{cdh}$, we will write $E(n)$ for $E\otimes \mathbb Z (n)$.

Notice that by construction 
the spectrum $H\mathbb Z \in \stablehomotopycdh$ 
representing motivic cohomology has transfers.  Furthermore, 
for every $m$, $n\in \mathbb Z$ \cite[5.10-5.11]{CD-1}:
\begin{align}  \label{eq.motcplx}
U(\Lambda(n)[m])\cong \Sigma ^{m,n} \mathbf L \pi^* H\Lambda.
\end{align}

Since $tr(\Sigma^{\infty}_T(X_+)) = M(X)$ for any $X \in \Sch_k$,
it follows that $M(Y) \to M(X) \to tr(\Sigma^{\infty}_T(M_f))$
is a distinguished triangle in $\DM_{cdh}$ for any closed immersion $Y \inj X$. 
We let $M(X/Y) := tr(\Sigma^{\infty}_T(M_f))$.

Let ${\rm DM}(k, \Lambda)$  and ${\rm DM}_{cdh}(k, \Lambda)$
denote Voevodsky's big categories of motives over $k$ with coefficients in
$\Lambda$ with respect to the Nisnevich and cdh-topologies, respectively.
Note that ${\rm DM}(k, \Z) = \DM_k$. 
It follows from \cite[Proposition 8.1(c)]{CD-1} that the functor
$\mathbf L \pi^*: {\rm DM}(k, \Lambda) \to {\rm DM}_{cdh}(k, \Lambda)$
is an equivalence of tensor triangulated categories. 
As a consequence of \ref{adj.DM-SH}-\ref{eq.motcplx} and 
\ref{eqn:RMC-1-0}-\ref{thm:cdhdescent}, 
we get:
\begin{prop}\label{prop:DMrep}
Let $k$ be a perfect field of exponential characteristic $p$.
For any closed immersion $f: Y \inj X$ in $\Sch_k$ and integers 
$m$, $n \in \Z$, there is a natural isomorphism
\[
H^m(X|Y, \Lambda(n))  \xrightarrow{\cong} 
 \Hom_{\DML}(M(X/Y), \Lambda(n)[m]).
\]
If $X, Y \in \Sm_k$, there is a natural isomorphism
\[
H^m(X|Y, \Z(n))  \xrightarrow{\cong} 
 \Hom_{\DM_k}(M(X/Y), \Z(n)[m]).
\]
\end{prop}

\subsection{Motivic Cohomology with (compact) supports}  
\label{ss.motcohcs}
Let $X\in \Sch _k$, $j:U\rightarrow X$ an open immersion and $Z=X\backslash U$
its closed complement equipped with the reduced scheme structure.  We will
write $M^c(X)\in \DM_{cdh}$, $M_Z(X)\in \DM_{cdh}$ respectively for the motive of
$X$ with compact supports \cite[\S 4.1]{Voev-1} and the motive of $X$
with supports on $Z$,
where the latter is defined in terms of a distinguished triangle
$M(U)\stackrel{j}{\rightarrow} M(X)\rightarrow M_Z(X)$ in $\DM_{cdh}$.

\begin{defn}  \label{df.motcohcs}
With the notation of \ref{eq.motcplx} and \ref{ss.motcohcs}, 
let $m$, $n\in \mathbb Z$.  The motivic
cohomology of $X$ with compact supports of degree $m$ and weight $n$ 
is given as:
\[  H^{m}_c(X,\mathbb Z(n)):=
\Hom _{\DM_{cdh}}(M^c(X),\mathbb Z(m)[n]).
\]
Similarly, the motivic cohomology of $X$ with supports on $Z$ of degree $m$ and
weight $n$ is given by $H^{m}_Z(X,\mathbb Z(n)):=
\Hom _{\DM_{cdh}}(M_Z(X),\mathbb Z(n)[m])$.
The motivic cohomology groups $H^{m}_c(X,\mathbb Z[\tfrac{1}{p}](n))$
and $H^{m}_Z(X,\mathbb Z[\tfrac{1}{p}](n))$ are defined analogously.
\end{defn}

\begin{prop}\label{rep.motcohcs}
With the notation of \ref{def.relKH}, assume that $f$ is a closed (resp. open)
immersion, and let $U=X\backslash Y$ (resp. $Z=X\backslash Y$).  
If $f$ is closed, assume in addition that $X$ is proper.
Then there are natural isomorphisms:
\begin{align*}  
\Hom_{\stablehomotopycdh}(\Sigma_T ^{\infty} M_f, 
& \Sigma ^{m,n} \mathbf L \pi^* H\Lambda) \cong 
H^{m}_c(U, \Lambda(n))\\
resp.\quad \Hom_{\stablehomotopycdh}(\Sigma_T ^{\infty} M_f, 
& \Sigma ^{m,n} \mathbf L \pi^* H\Lambda) \cong 
H^{m}_Z(X,\Lambda (n)).
\end{align*}

If $X, Y \in \Sm_k$, these isomorphisms hold integrally.
\end{prop}
\begin{proof}
We will give the argument for the first isomorphism since the other one is 
parallel.
By construction \eqref{def.relKH}, there is a distinguished triangle 
$(cs)=\Sigma _T^{\infty}Y_+ \rightarrow \Sigma _T^{\infty}X_+ \rightarrow 
\Sigma _T^{\infty}M_f$ in $\stablehomotopycdh$.  Combining the Quillen 
adjunctions
\eqref{adj.DM-SH} with \eqref{eq.motcplx}, we are reduced to show that
$tr(\Sigma _T^{\infty}M_f)\cong M^c(U)$ in $\DM_{cdh}$ in general and
in $\DM_k$ if $X, Y \in \Sm_k$.  But this follows from
\cite[Proposition~4.1.5]{Voev-1} and \cite[Proposition~8.1(c)]{CD-1} 
since applying $tr$ to (cs), we obtain a distinguished triangle 
$M(Y) \rightarrow M(X) \rightarrow tr(\Sigma _T^{\infty}M_f)$ in $\DM_{cdh}$
(in $\DM_k$ if $X, Y \in \Sm_k$).
\end{proof}

\subsection{Relative motivic cohomology in terms of 
hypercohomology of sheaves of equidimensional cycles}\label{sec:RMC}
Let $X \in \Sm_k$ and let $Y \subset X$ be an effective Cartier divisor.
In this case, we can describe the relative motivic cohomology group in terms of
the hypercohomology of a complex of equidimensional cycles.
Let $S_X$ denote the double of $X$ along $Y$ and let $\nabla: S_X \to
X$ denote the fold map (see \S~\ref{sec:double}).
We let $\nabla$ also denote the morphism of sites 
$(S_X)_{cdh} \rightarrow X_{cdh}$. For any $Z\in \Sch _k$, let $Sh(Z)_{cdh}$ 
denote the category of sheaves of abelian groups on the cdh-site $Z_{cdh}$.

\begin{lem}\label{lem:hyper-0}
The direct image functor $\nabla _{\ast}: Sh(S_X)_{cdh} \rightarrow Sh(X)_{cdh}$ 
is exact. 
\end{lem}
\begin{proof}
It suffices to show that $\nabla :S_X \rightarrow X$ is a cdh-covering.  
But this follows immediately from \cite[Lemma 5.8]{SV} because $\nabla$
is finite and it admits a section (see \S~\ref{sec:double}).
\end{proof}

For any $Z \in \Sch_k$ and an integer $n \ge 0$, we let $\Z_Z(n)[2n]$
denote the complex of sheaves $C_*z_{equi}(\A^n_k,0)_{Z_{cdh}}$ on
$Z_{cdh}$. 
It follows from \lemref{lem:hyper-0} that
\begin{equation}\label{eqn:nabla*-0}
H^m(S_X, \Lambda(n)) \cong \H^{m-2n}_{cdh}(X, \nabla_*(\Lambda_{S_X}(n)[2n])).
\end{equation}

For any morphism of schemes $Y \to X$, we obtain
maps $Y \xrightarrow{\iota_-} S_X \times_X Y \to Y$ whose composite is 
identity.
Pulling back cycles along $\iota_-$, we obtain a
map $\iota^*_-: \nabla_*(\Lambda_{S_X}(n)[2n]) \to \Lambda_{X}(n)[2n]$ which
admits a section 
$\nabla^*: \Lambda_{X}(n)[2n] \to \nabla_*(\Lambda_{S_X}(n)[2n])$.
We let $\Lambda_{X|Y}(n)[2n] := {\rm Ker}(\iota^*_-)$ so that there is
a split short exact sequence on $X_{cdh}$:
\begin{equation}\label{eqn:nabla*}
0 \to \Lambda_{X|Y}(n)[2n] \to \nabla_*(\Lambda_{S_X}(n)[2n]) \to 
\Lambda_{X}(n)[2n] \to 0.
\end{equation}

Since ~\eqref{eqn:rel-et-2} is a $cdh$-square (see \S~\ref{sec:double}), the
map $\iota^*_+: 
H^m(S_X|X_-, \Lambda(n)) \to H^m(X|Y, \Lambda(n))$ is an isomorphism.
We conclude from ~\eqref{eqn:nabla*-0} and ~\eqref{eqn:nabla*}, the 
following description of the relative motivic cohomology of the pair
$(X,Y)$ in terms of equidimensional cycles.

\begin{prop}\label{prop:rel-mot-coh-equid}
For any integers $m, n \in \Z$, there is a natural isomorphism
\[
H^m(X|Y, \Lambda(n)) \cong \H^{m-2n}_{cdh}(X, \Lambda_{X|D}(n)[2n]).
\]
If $X$ is furthermore projective, we have
\[
H^m(X|Y, \Lambda(n)) \cong 
\H^{m-2n}_c(X \setminus Y, (C_{\ast}z_{equi}(\A^n_k,0)_{\Lambda})_{cdh}).
\]
\end{prop}

\section{Slice spectral sequence for relative $KH$-theory}
\label{sec:SSS}
Let $k$ be a perfect field of exponential characteristic $p$.
Given $X \in \Sch_k$, recall that Voevodsky's slice filtration of
$\stablehomotopy _{X}$ is given as follows.
For an integer $q \in \Z$, let $\Sigma^{q}_{T}\stablehomotopy _{X}^{\fnteff}$ 
denote the smallest full triangulated subcategory of $\stablehomotopy _{X}$ 
which contains $C^{q}_{\fnteff}$ and is closed under arbitrary coproducts, where
\begin{equation}\label{eq:effcat}
C^{q}_{\fnteff}=\{ \Sigma ^{m,n}\Sigma _T ^{\infty}Y_{+}: m\in \mathbb Z, 
n\geq q, Y\in \Sm_X\}.  
\end{equation}

In particular, $\stablehomotopy _{X}^{\fnteff}$ is the smallest 
full triangulated subcategory of $\stablehomotopy _{X}$ which is closed
under infinite direct sums and contains all spectra of the type
$\Sigma^{\infty}_TY_+$ with $Y \in \Sm_X$.
The slice filtration of $\stablehomotopy _{X}$ (see \cite{Voev-2}) is the
sequence of full triangulated subcategories
\[ 
\cdots \subseteq \susp{q+1}_{T}\stablehomotopy _{X}^{\fnteff} \subseteq
\susp{q}_{T}\stablehomotopy _{X}^{\fnteff} \subseteq \susp{q-1}_{T}
\stablehomotopy _{X}^{\fnteff} \subseteq \cdots
\]

It is known \cite{Pelaez-3} that the 
inclusion $i_{q}:\susp{q}_{T}\stablehomotopy _{X}^{\fnteff}\rightarrow 
\stablehomotopy_X$ admits a right adjoint 
$r_{q}:\stablehomotopy_X \rightarrow \susp{q}_{T}\stablehomotopy _{X}^{\fnteff}$
and that the functors 
$f_{q}, s_{<q}, s_{q}:\stablehomotopy_X \rightarrow \stablehomotopy_X$
are triangulated; where $r_q \circ i_q$ is the identity,
$f_{q}=i_{q}\circ r_{q}$ and $s_{<q}$, $s_{q}$ are 
characterized by the existence of the following distinguished triangles in 
$\stablehomotopy _{X}$:
\begin{equation}\label{eqn:Slice-triangle}
\xymatrix@R=0.8pc{f_{q}E\ar[r] & E \ar[r]& s_{<q}E \ \ \ \mbox{and}  & 
			f_{q+1}E\ar[r]& f_{q}E \ar[r]& s_{q}E}
\end{equation}
for every $E\in \stablehomotopy _{X}$.

\begin{defn}  \label{def.indfil}
Let $a$, $b$, $n \in \mathbb Z$ and $Y\in \Sm_X$.
Let $F^{n}E^{a,b}(Y)$ be the image of the map induced by $f_{n}E\rightarrow E$
\eqref{eqn:Slice-triangle}:
$\Hom _{\stablehomotopy _X}(\Sigma _T ^\infty Y_+, \Sigma ^{a,b}f_n E)\rightarrow
\Hom _{\stablehomotopy _X}(\Sigma _T ^\infty Y_+, \Sigma ^{a,b}E)$.  This 
determines a
decreasing filtration $F^{\bullet}$ on $E^{a,b}(Y)=\Hom _{\stablehomotopy _X}
(\Sigma _T ^\infty Y_+, \Sigma ^{a,b}E)$, and we will write $gr^{n}F^\bullet$ for 
the associated graded pieces
$\{F^{n}E^{a,b}(Y)/F^{n+1}E^{a,b}(Y)\}$.
\end{defn}

\subsection{The slice spectral sequence}\label{sec:SS-Conv}
Let $Y\in \Sm_X$ be a smooth $X$-scheme and $G \in \stablehomotopy _{X}$.  
Since $\stablehomotopy _{X}$ is a triangulated category, the collection of 
distinguished triangles
$\{ f_{q+1}G\rightarrow f_{q}G\rightarrow s_{q}G\}_{q\in \mathbb Z}$ determines a 
(slice) spectral sequence of the form
$E_1^{p,q}=\Hom _{\stablehomotopy _{X}}(\Sigma _T ^{\infty}
Y_+,\Sigma _s ^{p+q}s_p G)$ 
with $G^{m, n}(Y)=
\Hom _{\stablehomotopy _{X}}(\Sigma^{\infty}_TY_{+},\Sigma^{m,n}  G)$ 
as its abutment and differentials 
$d_r: E_r^{p,q}\rightarrow E_r^{p+r, q-r+1}$.

In order to study the convergence of this spectral sequence,
recall from \cite[p.~22]{Voev-2} that $G\in \stablehomotopy _{X}$ is called 
{\sl bounded} with respect to the slice filtration if for every 
$m$, $n\in \mathbb Z$ and  every $Y\in \Sm_X$, there exists 
$q\in \mathbb Z$ such that:
\begin{equation}  \label{eqn:bound}
\Hom _{\stablehomotopy _{X}}(\susp{m,n}\Sigma^{\infty}_TY_{+}, f_{q+i}G)=0
\end{equation}
for every $i>0$.  Clearly the slice spectral sequence 
is strongly convergent when $G$ is bounded.

\subsection{The slice spectral sequence for $\KGL$}\label{sec:KH-theory}
Let $k$ be a perfect field of exponential characteristic $p$, and let
$f: Y \to X$ be a morphism in $\Sch _k$.
Let $v : X \to \Spec(k)$ be the structure map.
Recall that $\KGL$ is by definition $\mathbf L \pi ^{\ast} \KGL _k$, and
the map $\mathbf L v^*(\KGL_k) \to \KGL_X$ is an isomorphism
by \cite[Proposition~3.8]{Cisinski}. 

By \cite[6.2.3.10, 5.3.18 and 5.3.10]{Riou} we have the
Chern character isomorphism
$ch: (\KGL_k)_{\Q} \xrightarrow{\cong} \oplus_{q\in \Z} \Sigma _T^q H\Q$,
and by \cite[5.3.17 and 5.3.10]{Riou}
$\Sigma _T^q H\Q \cong (\KGL _k)_{\mathbb Q}^{(q)}$ where
$(\KGL _k)_{\mathbb Q}^{(q)}$ is the $q$th Adams eigen-spectrum
constructed in \cite[5.3.9]{Riou}.  In addition, one knows that 
$\phi_r:s_r \KGL_k \xrightarrow{\cong} \Sigma^r_T H\Z$ for $r \in \Z$ 
\cite[Theorem~6.4.2]{Levine-3}, which implies 
$s_i(\Sigma^r_T H\Z)=0$ for $i\neq r$ and
$s_r(\Sigma^r_T H\Z)\cong \Sigma^r_T H\Z$.  Now,  since the effective
categories $\Sigma^{q}_{T}\stablehomotopy ^{\fnteff}$
are closed under arbitrary coproducts it follows
that the slices $s_i$ commute with arbitrary coproducts.  Hence,
we obtain the following isomorphism: 
\[
s_i(ch): s_i(\KGL_k)_{\Q} \xrightarrow{\cong} s_i(\oplus_{q\in \Z} \Sigma _T^q H\Q)\cong \oplus_{q\in \Z} s_i\Sigma _T^q H\Q
\cong  \Sigma _T^i H\Q \cong (\KGL _k)_{\mathbb Q}^{(i)}.
\]

Thus, we conclude that the Chern character isomorphism
splits the slice filtration:
\begin{equation}\label{eqn:Chern-iso}
{\rm ch}: (\KGL_k)_{\Q} \xrightarrow{\cong} \oplus_{q\in \Z} 
s_q(\KGL_k)_{\Q}.
\end{equation}
in $\stablehomotopy$.
For $q \in \Z$, we 
\nolinebreak
let
\begin{equation}\label{eqn:layer}
\begin{array}{lll} 
\Z_{\rm sl}(M_f, q): & = & \mathbf R\Gamma _k \circ \mathbf{RHom} 
(\Sigma _{S^1}^{\infty}M_f, 
\mathbf R\Omega _{t}^{\infty} \Sigma^q_T H\Z) \in \mathcal H (S ^1).
\end{array}
\end{equation}

If $f$ is a closed (resp. open) immersion, we write $\Z_{\rm sl}(M_f, q)$ as
$\Z_{\rm sl}(X|Y, q)$ (resp. $\Z_{\rm sl}^Y(X, q))$.

Now, by constructions \ref{def.relKH}\eqref{def.relKH.b} and 
\eqref{dg.map.cofs}, there is a commutative diagram in
$\Tspectracdh$:
\begin{align}  \label{dg.Tcofseq}
\begin{array}{c}
\xymatrix@R=1pc{\Sigma _T^{\infty}Y_+ \ar[r] \ar[d]_-f & 
\Sigma _T^{\infty}A_f \ar[r]
	\ar[dl]^-{w_f}& \Sigma _T^{\infty}M_f \\
	\Sigma _T^{\infty}X_+ &&} 
\end{array}
\end{align}
where the top row is a cofiber sequence.  Since $w_f:A_f\rightarrow X_+$ is a 
weak equivalence of motivic spaces, the map
$w_f:\Sigma _T ^{\infty}A_f\rightarrow \Sigma _T ^{\infty}X_+$ 
is a stable 
weak equivalence in $\Tspectracdh$ (e.g., see \cite[page~592]{Voev-0}).
Hence, mapping in $\stablehomotopycdh$ \eqref{dg.Tcofseq} into the slice 
tower of $\KGL$:
\[ \cdots \rightarrow \mathbf L \pi ^{\ast}(f_{q+1}\KGL_k) \rightarrow 
\mathbf L \pi ^{\ast}(f_q\KGL_k) \rightarrow \cdots \rightarrow \mathbf L 
\pi ^{\ast}(\KGL_k) = \KGL
\]
and splicing together \ref{lm.secpres} and \ref{cor.hfibident},
we obtain the following result.

\begin{thm}\label{thm:KH-SS*}
Let $k$ be a perfect field, and let $f:Y\rightarrow X$ be a morphism 
in $\Sch _k$. Then there exists in $\mathcal H(S^1)$,
a tower, natural in $(X,Y)$: 
\begin{equation}\label{eqn:KH-SS-0}
\cdots \to \phi_{q+1} KH(f) \to \phi_{q} KH(f) \to \cdots \to
\phi_0 KH(f) \to \phi_{-1} KH(f) \to \cdots \to KH(f)
\end{equation}
and an isomorphism for each $q \in \Z$:
\begin{equation}\label{eqn:KH-SS-1}
{\phi_{q}}/{\phi_{q+1}} KH(f)_{\Lambda} \cong
\Lambda_{\rm sl}(M_f, q).
\end{equation}
where ${\phi_{q}}/{\phi_{q+1}} KH(f)$ is the cofiber in $\mathcal H (S ^1)$
of the map $\phi _{q+1}KH(f)\rightarrow \phi _{q}KH(f)$.
\end{thm}
\begin{proof}
If we let 
$\phi_{q} KH(f) = \mathbf R\Gamma _k \circ \mathbf{RHom} 
(\Sigma _{S^1}^{\infty}M_f, 
\mathbf R\Omega _{t}^{\infty} \mathbf L \pi ^{\ast}(f_q\KGL_k))$, 
we only need to verify that the
cofiber of the map $\phi_{q+1} KH(f)_{\Lambda} \to \phi_{q} KH(f)_{\Lambda}$
is isomorphic to $\Lambda_{\rm sl}(M_f, q)$. But this follows from the fact that
there is an isomorphism $\Sigma^q_TH\Z \cong {f_q\KGL_k}/{f_{q+1}\KGL_k}$
in $\sS\sH$ (e.g., see \cite[\S~11.3, Remark~11.3.4]{Levine-3}).
Moreover, the same holds for $\KGL$ in $\stablehomotopycdh$ with 
$\Lambda$-coefficients
in positive characteristic by \cite[Lemma~4.8]{KP}.
\end{proof}

Using \thmref{thm:KH-SS*} together with the Quillen adjunctions of
\eqref{eq.s1Tadj} and \eqref{eq.globalsec}, we get our final result:

\begin{thm}\label{thm:KH-SS}
Let $k$ be a perfect field and let $f:Y\rightarrow X$ be a morphism
in $\Sch _k$.
Then there is a commutative diagram of
strongly convergent spectral sequences:
\[\xymatrix@R=1pc{\vdots \ar[d] & \vdots \ar[d]\\
	E^{a,b}_2 = H^{a-b}(M_f, \Lambda(-b))
\ar@{=>}[r]  \ar[d]& KH_{-a-b}(f)_{\Lambda} \ar[d]\\
	E^{a,b}_2 = H^{a-b}(X, \Lambda(-b)) 
\ar@{=>}[r] \ar[d]^-{f^\ast} & KH_{-a-b}(X)_\Lambda \ar[d]^-{f^\ast}\\ 
	E^{a,b}_2 = H^{a-b}(Y, \Lambda(-b))
\ar@{=>}[r] \ar[d]& KH_{-a-b}(Y)_{\Lambda} \ar[d] \\
	E^{a+1,b}_2 = H^{a-b+1}(M_f, \Lambda(-b))
\ar@{=>}[r] \ar[d]& KH_{-a-b-1}(f)_{\Lambda} \ar[d] \\
\vdots  & \vdots}
\]
with exact columns, 
where the differentials of the spectral sequence are given by
$d_r: E^{a,b}_r \to E^{a+r, b-r+1}_r$; and for every $a$, $b\in \mathbb Z$,
there exists $N>0$ such that $E^{a,b}_r=E^{a,b}_\infty$ for $r\geq N$,
where $E^{a,b}_\infty$
is the associated graded $gr^{-b}F^{\bullet}$ with respect
to the descending filtration \ref{def.indfil} on the groups in the right 
column.
Furthermore, these spectral sequences degenerate with rational coefficients.
\end{thm}
\begin{proof}
Except for the strong convergence, all the claims follow from 
\thmref{thm:KH-SS*}, \eqref{eq.s1Tadj} and \eqref{eq.globalsec} and an 
elementary reindexing to convert the resulting $E_1$-spectral
sequence into an $E_2$-spectral sequence (see the proof of
\cite[Theorem~11.3.2]{Levine-3}).
Alternatively, one can use 
\ref{lm.secpres} and \corref{cor.hfibident} and the 
method of \cite[4.27]{KP}) to prove the existence of the spectral sequences.

To prove the strong convergence, we recall from \cite[4.7 and 4.10]{KP} that
$\KGL[\tfrac{1}{p}]$ is bounded with respect to the slice filtration 
\eqref{eqn:bound}.
Let $m$, $n\in \mathbb Z$ be arbitrary integers, and $q_1$, $q_2\in \mathbb Z$
such that the vanishing condition \eqref{eqn:bound} holds for $(m,n)$ and 
$(m+1,n)$
respectively.  Then, if $q$ is the maximum of $q_1$, $q_2$, we deduce by 
\cite[Theorem~2.14]{KP} and \eqref{eqn:bound} that for every $i>0$, there is 
an exact sequence:
\[ 
\xymatrix@R=1pc@C=1pc{\Hom _{\stablehomotopy _{cdh}}(\susp{m+1,n}
\Sigma^{\infty}_T
Y_{+}, f_{q+i}\KGL[\tfrac{1}{p}]) \ar[r] & \Hom _{\stablehomotopy _{cdh}}(\susp{m,n}
\Sigma^{\infty}_T M_f, f_{q+i}\KGL[\tfrac{1}{p}]) \ar[d] \\
& \Hom _{\stablehomotopy _{cdh}}(\susp{m,n}\Sigma^{\infty}_TX_{+}, 
f_{q+i}\KGL[\tfrac{1}{p}]),}
\]
where the terms on the left and bottom vanish, so the term in the middle also 
vanishes.
This implies the desired convergence abutting to
$KH_\ast (f)_{\Lambda}$.
\end{proof}

\begin{remk}  \label{rm.sscsup}
With the notation of Definition~\ref{def.relKH}, we have the following.
\begin{enumerate}
\item \label{rm.sscsup.a} If $f$ is an open immersion, then 
Corollary~\ref{cor.hfibident} and \propref{rep.motcohcs} together imply that
the top row of \ref{thm:KH-SS} is same as (where $Z=X\backslash Y$):
\begin{align} \label{eq.sssup}
E_2^{a,b}=H^{a-b}_Z(X, \Lambda(-b)) \Rightarrow KH^Z_{-a-b}(X)_{\Lambda}.
\end{align}
\item \label{rm.sscsup.b} If $f$ is a closed immersion, then 
Corollary~\ref{cor.hfibident} and  \ref{eqn:RMC-1-0}-\ref{thm:cdhdescent} 
together imply that
the top row of \ref{thm:KH-SS} is same as
\[ E_2^{a,b} = H^{a-b}(X|Y, \Lambda(-b)) \Rightarrow KH_{-a-b}(X, Y)_{\Lambda}.
\]
If in addition $X$ is assumed to be proper, \corref{cor.hfibident} 
and \propref{rep.motcohcs} together imply that the top row of 
\thmref{thm:KH-SS} 
is same as
\begin{align} \label{ss.csupp}
E_2^{a,b}=H^{a-b}_c(X \setminus Y,\Lambda(-b)) 
\Rightarrow KH_{-a-b}(X,Y)_{\Lambda}.
\end{align}
\end{enumerate}
\end{remk}

As a combination of ~\thmref{thm:KH-SS} and \cite[Theorem~9.5, 9.6]{TT}, 
we obtain the following result for the Thomason-Trobaugh relative algebraic 
$K$-theory $K(-)$ of singular schemes \cite{TT}.

\begin{cor}\label{cor:Alg-K-SS}
Let $k$ be a perfect field of exponential characteristic $p$. Let
$\ell \neq p$ be a prime and $m \ge 0$ any integer. 
Given any $f: Y \to X$ in $\Sch_k$, there exists a strongly convergent 
spectral sequence
\begin{equation}\label{eqn:K-SS-1}
E^{a,b}_2 = H^{a-b}(M_f, \Z/{\ell^m} (-b))
\Rightarrow K/{\ell^m}_{-a-b}(f).
\end{equation}

If $p > 0$, there exists a strongly convergent spectral sequence
\begin{equation}\label{eqn:K-SS-0}
E^{a,b}_2 = H^{a-b}(M_f, \Z[\tfrac{1}{p}](-b))
\Rightarrow K_{-a-b}(f)[\tfrac{1}{p}].
\end{equation}
\end{cor}

\subsection{Relative $K$-theory of smooth closed pair}
\label{sec:Rel-sm}
Our main example of interest is the relative algebraic $K$-theory
of the pair $(X,Y)$ where $Y \subset X$ is a codimension one
closed immersion of smooth schemes.
We now show that Theorems~\ref{thm:KH-SS*} and ~\ref{thm:KH-SS} 
can be much simplified to give a more precise description of
the tower of the relative $K$-theory and of the motivic cohomology in this
case.

In any case, it follows from the definition of the tower 
~\eqref{eqn:KH-SS-0} that there is a distinguished triangle in $\mathcal H (S^1)$:
\[
\phi_q KH(f) \to \phi_q KH(X) \xrightarrow{f^*} \phi_q KH(Y)
\to \phi_q KH(f)[1].
\]

Let us now assume that $f$ is a closed immersion of smooth schemes over $k$. 
In this case, we know that $K(f) \to KH(f)$ is a weak equivalence.
On the other hand, it follows from \cite[Theorem~9.0.3]{Levine-3} that
the slice tower $\phi_{\bullet} K(X)$ coincides with the homotopy coniveau
tower $\psi_{\bullet} K(X)$, where $\psi_q K(X)$ is the diagonal of the 
simplicial spectrum $K^{(q)}(X, -)$, defined in  \cite[2.1.2]{Levine-3}.
In particular, we get $\phi_q K(f) = \phi_q K(X,Y) = K(X,Y)$ for $q \le 0$.
Furthermore, it follows from \thmref{thm:KH-SS*} and
\cite[Theorem~6.4.2]{Levine-3} that
$\Z_{\rm sl}(X, q) \cong {\phi_q}/{\phi_{q+1}} K(X) \cong 
{\psi_q}/{\psi_{q+1}} K(X) \cong \Sigma^{\infty}_{S^1}z^q(X, \bullet)$, where 
$z^q(X, \bullet)$ is 
Bloch's cycle complex of $X$ \cite{Bloch}. The same holds for $Y$ as well.

We let $z^q_{\sM}(X|Y, \bullet) = {\rm Cone}(f^*)[-1]$, where
${\rm Cone}(f^*)$ denotes the cone of the restriction map of cycle
complexes $f^*: z^q_Y(X, \bullet) \to z^q(Y, \bullet)$. Recall here that
$z^q_Y(X, \bullet) \subset z^q(X, \bullet)$ is the subcomplex generated by
irreducible cycles which intersect all faces of $Y \times \Delta^{\bullet}$
properly. 
By \cite[Thm. 1.10]{Krishna-Levine} (which relies on
Bloch's moving \cite{Bloch-2} lemma) this inclusion
is a weak equivalence of simplicial abelian groups. 
Letting $\CH^q_{\sM}(X|Y, i) = H_i(z^q_{\sM}(X|Y, \bullet))$ and
$\CH^{*}_{\sM}(X|Y, i) = \oplus_{q \in \Z} \CH^q_{\sM}(X|Y, i)$,
we conclude that Theorem~\ref{thm:KH-SS*} reduces to 
the following.

\begin{thm}\label{thm:sm-pair}
Let $k$ be a perfect field and let $f: Y \inj X$ be a closed immersion in
$\Sm_k$. Let $(g, g'): 
(X', Y') \to (X,Y)$ be a projective morphism of closed pairs 
in $\Sm_k$ such that $Y' = Y \times_X X'$. Let $T_g \in K_0(Y')$ be the
virtual relative tangent bundle for $g:X'\rightarrow X$. 
Then there is a tower in $\mathcal H (S^1)$:
\[
\cdots \to \phi_{q+1} K(X,Y) \to \phi_{q} K(X,Y) \to \cdots \to
\phi_0 K(X,Y) = K(X,Y)
\]
and isomorphisms for each $q, i \ge 0$:
\begin{enumerate}
\item
${\phi_q}/{\phi_{q+1}} K(X, Y) \cong
\Sigma^{\infty}_{S^1}z^q_{\sM}(X|Y, \bullet)$ 
\item
$\CH^q_{\sM}(X|Y, i) \cong H^{2q-i}(X|Y, \Z(q))$.
\item
There exists a strongly convergent spectral sequence
\[
E^{a,b}_2 = \CH^{-b}_{\sM}(X|Y, -a-b) 
\Rightarrow KH_{-a-b}(X,Y).
\]
\item
Grothendieck Riemann-Roch theorem: there is a commutative diagram
\begin{equation}\label{eqn:RR-iso}
\xymatrix{
K_i(X', Y')_{\Q} \ar[rr]^-{{\rm Td}(T_g){\rm ch}} \ar[d]_{g_*} && \CH^{*}_{\sM}(X'|Y', i)_{\Q} 
\ar[d]^{g_\ast} \\
K_i(X, Y)_{\Q} \ar[rr]^-{{\rm ch}}  && \CH^{*}_{\sM}(X|Y, i)_{\Q},}
\end{equation}
such that the horizontal arrows are isomorphisms.
\end{enumerate}
\end{thm} 
\begin{proof}
The existence of the tower and (1) are already explained above.
The item (3) follows from (1) and (2). 
For (2), we note from what is explained above and the Quillen adjunctions
~\eqref{eq.s1Tadj} and ~\eqref{eq.globalsec} that we do have
natural maps 
\begin{equation}\label{eqn:cycl-map-3-high-0}
\CH^q_{\sM}(X|Y, i) \to 
\H^{-i}(X, z^q_{\sM}(-|Y, \bullet)|_{X_{zar}}) \to
H^{2q-i}(X|Y, \Z(q)).
\end{equation}

To show that the composite map is an isomorphism, we only have to observe that 
this is indeed the case if $Y = \emptyset$. We conclude 
using the 5-lemma. In the Riemann-Roch theorem (4), the
commutativity of the diagram follows from
the item (2) and \cite[Theorem~3.15]{Nav},
if we note that a projective morphism of smooth schemes is a local
complete intersection.  Finally $ch$ is an isomorphism by
\eqref{eqn:Chern-iso},   which also implies that
${\rm Td}(T_g){\rm ch}$ is an isomorphism since $Td(T_g)$
is a unit in $CH^\ast_{\mathcal M}(X',\ast)_{\mathbb Q}$ and
$CH^\ast_{\mathcal M}(X'|Y',\ast)_{\mathbb Q}$
has the structure of $CH^\ast_{\mathcal M}(X',\ast)_{\mathbb Q}$-module.
\end{proof}

\section{The cycle class and Chern class maps}\label{sec:Conn-K}
As an application of \thmref{thm:KH-SS}, we shall
now construct a cycle class map from the
relative motivic cohomology in the 0-cycle range 
to relative $KH$-theory. We shall then apply the double construction
to construct the Chern class maps.


\subsection{The cycle class map}\label{sec:CCM}
We continue with our assumption on the field $k$ and the coefficient ring 
$\Lambda$. In order to construct the cycle class map, we shall use the 
connective version of the spectrum $\KGL$.

Let $X \in \Sch_k$. Recall that the {\sl connective $K$-theory} spectrum
$\KGL^0_X$ is the motivic $T$-spectrum $f_0\KGL_X$
in $\stablehomotopy_X$ (see ~\eqref{eqn:Slice-triangle}).  
In particular, there is a canonical map $u_X: \KGL^0_X \to \KGL_X$
which is universal for morphisms from objects of 
$\stablehomotopy _{X}^{\fnteff}$ to $\KGL_X$. 
With the notation of Definition~\ref{def.relKH}, we let 
$CKH^{p,q}(M_f) = \Hom_{\stablehomotopycdh}(\Sigma^{\infty}_T M_f, 
\Sigma^{p,q} \mathbf L \pi ^{\ast}\KGL^0)$.

Apart from the connective cover of $\KGL$, we also need the following
vanishing result for the motivic cohomology with support.

\begin{lem}\label{lem:vanish-coh}
Let $X \in \Sch_k$ be of dimension $d$ and let
$Y \subseteq X$ be a closed subscheme. Then
$H^{2a-b}_Y(X, \Lambda(a)) = H^{2a-b}_c(X, \Lambda(a)) = 0$ whenever $a > d+b$.
\end{lem}
\begin{proof}
We first show that $H^{2a-b}_Y(X, \Lambda(a)) = 0$ whenever $a > d+b$.
If $(X,Y)$ is a smooth pair, then the purity theorem for motivic
cohomology (see \cite[Theorem~4.10]{SV}) implies that
$H^{2a-b}_Y(X, \Lambda(a)) \cong H^{2(a-q) -b}(Y, \Lambda(a-q))$, 
where $Y$ has codimension $q$ in $X$.
But it is shown in \cite[Theorem~5.1]{KP} that
the latter group is zero since $a-q > d-q+b = \dim(Y) + b$.

We shall now prove the lemma by induction on the dimensions of $X$ and $Y$.
We first keep our assumption that $X$ is smooth but allow $Y$ to be singular.
If $\dim(Y) = 0$, then we can assume that $Y$ is smooth and reduce to
the previous case. We can therefore assume that $\dim(Y) \ge 1$.
Let $Z$ be the singular locus of $Y$ with reduced induced closed subscheme
structure. Since $k$ is perfect, we have $\dim(Z) < \dim(Y)$.
There is a commutative diagram
\begin{equation}\label{eqn:vanish-coh-0}
\begin{array}{c}
\xymatrix@C=.8pc@R=1pc{
M(X \setminus Y) \ar[r] \ar[d] & M(X) \ar[r] \ar@{=}[d] & M_Y(X) \ar[d] \\
M(X \setminus Z) \ar[r] \ar[d] & M(X) \ar[r] & M_Z(X) \\
M_{Y\setminus Z}(X \setminus Z) & &}
\end{array}
\end{equation}
in $\DM_k$ so that we get a distinguished triangle 
$M_{Y\setminus Z}(X \setminus Z)  \to M_Y(X) \to M_Z(X)$.
This yields an exact sequence
\[
H^{2a-b}_Z(X, \Lambda(a)) \to H^{2a-b}_Y(X, \Lambda(a))  \to 
H^{2a-b}_{Y\setminus Z}(X \setminus Z, \Lambda(a)).
\]
The first term vanishes by induction on $\dim(Y)$ and the third term
vanishes because $(X \setminus Z, Y\setminus Z)$ is a smooth pair.
Hence the middle term vanishes.

We now allow $X$ to be singular and work by induction on $\dim(X)$.
If $\dim(X) = 0$, then we can assume $X$ to be smooth. So we assume
$\dim(X) \ge 1$. Let us first assume that
there is a resolution of singularities
$f: \wt{X} \to X$ and let $\wt{Y} = f^{-1}(Y)$. Let $Z$ denote the singular
locus of $X$ and let $E \subset \wt{X}$ be the exceptional divisor.
If we let $U = X \setminus Y$ and $\wt{U} = f^{-1}(U)$, we get a
commutative diagram where the rows are distinguished triangles in $\DM_k$:
\begin{equation}\label{eqn:vanish-coh-1}
\begin{array}{c}
\xymatrix@C=.8pc@R=1pc{
M(U \cap Z) \oplus M(\wt{U}) \ar[r] \ar[d] & M(U) \ar[r] \ar[d] &
M(E \cap \wt{U})[1] \ar[d] \\
M(Z) \oplus M(\wt{X}) \ar[r] & M(X) \ar[r] & M(E)[1]}
\end{array}
\end{equation}
and this gives a distinguished triangle
\[
M_{Y \cap Z}(Z) \oplus M_{\wt{Y}}(\wt{X}) \to M_Y(X) \to 
M_{\wt{Y} \cap E}(E)[1].
\]
The associated long exact sequence of motivic cohomology groups is of the
form
\[
H^{2a-b-1}_{\wt{Y} \cap E}(E, \Lambda(a)) \to
H^{2a-b}_Y(X, \Lambda(a)) \to H^{2a-b}_{Y\cap Z}(Z, \Lambda(a)) \oplus
H^{2a-b}_{\wt{Y}}(\wt{X}, \Lambda(a)).
\]

\enlargethispage{25pt}

The end terms vanish either by induction on $\dim(X)$ or
by the case of smooth ambient scheme. It follows that the middle term
vanishes, as desired.

If $X$ is not smooth and $k$ has positive characteristic, 
we argue as follows.  By a theorem of Gabber 
\cite{Kelly} and its strengthening by Temkin 
\cite[Theorem~1.2.9]{Temkin}, there exists $W\in \Sm _k$
and a surjective proper map $h:W\rightarrow X$, which is generically \'etale 
of degree
$p^r$, $r\geq 1$.  By a theorem of Raynaud-Gruson
\cite[Theorem~5.2.2]{Ray}, there exists a blow-up $f:\wt{X} \rightarrow X$ 
with 
nowhere dense center $Z \subset X$ 
such that the following diagram commutes, where $h'$ is finite flat surjective 
of degree $p^r$
and $f'$ is the blow-up of $W$ with center $h^{-1}(Z)$:
\begin{align}  \label{bwupsqr.1}
\begin{array}{c}
\xymatrix@R=1pc{\wt{W} \ar[d]_-{f'} \ar[r]^-{h'}& \wt{X} \ar[d]^-{f} \\
	W \ar[r]_-{h} & X.}
\end{array}
\end{align}

Using the $cdh$-excision corresponding to the $cdh$-cover 
$\{\wt{X} \amalg Z \rightarrow X\}$ of $X$ and
arguing as above, we get an exact sequence:
\begin{align*}
H^{2a-b-1}_{\wt{Y} \cap E}(E, \Lambda(a)) \rightarrow 
H^{2a-b}_Y(X, \Lambda(a)) \rightarrow 
H^{2a-b}_{\wt{Y}}(\wt{X},\Lambda(a)) \oplus 
H^{2a-b}_{Y\cap Z}(Z, \Lambda(a)),
\end{align*}
where we set $E = X' \times _{X} Z$ and ${\wt{Y}}=f^{-1}(Y)$.
Since $\dim(E)$ and $\dim(Z)$ are smaller than $\dim(X)$,
it follows by induction on $\dim(X)$ that 
$H^{2a-b-1}_{\wt{Y} \cap E}(E, \Lambda(a))=0$ and
$H^{2a-b}_{Y\cap Z}(Z, \Lambda(a))=0$, thus
the map 
$f^*: H^{2a-b}_{Y}(X,\Lambda(a)) \to 
H^{2a-b}_{\wt{Y}}(\wt{X},\Lambda(a))$ is injective.

We now let $Z' = h^{-1}(Y), \ \wt{Z} = f'^{-1}(Z')$ and
consider the commutative diagram resulting from ~\eqref{bwupsqr.1}:
\begin{equation}\label{eqn:vanish-coh-2}
\begin{array}{c}
\xymatrix@R=1pc{
H^{2a-b}_{Y}(X,\Lambda(a)) \ar[r]^-{f^*} \ar[d]_{h^*} &
H^{2a-b}_{\wt{Y}}(\wt{X},\Lambda(a)) \ar[d]^{h'^*} \\
H^{2a-b}_{Z'}(W,\Lambda(a)) \ar[r]_-{f'^*} &
H^{2a-b}_{\wt{Z}}(\wt{W},\Lambda(a)).}
\end{array}
\end{equation} 

Since $W \in \Sm_k$, we know that $H^{2a-b}_{Z'}(W,\Lambda(a)) = 0$.
Using ~\eqref{eqn:vanish-coh-2}, it suffices therefore to show
that $h'^*$ is injective.  But this follows from \lemref{lem:coh-inj}.
The proof of the first vanishing assertion is now complete.

To prove that $H^{2a-b}_c(X, \Lambda(a)) = 0$ whenever $a > d+b$,
we choose an open immersion $j: X \inj \ov{X}$ with dense image
such that $\ov{X}$ is projective over $k$. 
Letting $Z = \ov{X} \setminus X$, it follows from \cite{Voev-1} and
\cite[Chapter~5]{Kelly} that there is a distinguished triangle in 
${\rm DM}(k, \Lambda)$:
\begin{equation}\label{eqn:loc-seq}
M^c(Z) \to M^c(\ov{X}) \to M^c(X) \to M^c(Z)[1].
\end{equation}

In particular, there is an exact sequence
\[
H^{2a-b-1}(Z, \Lambda(d)) \to H^{2a-b}_c(X, \Lambda(a)) \to
H^{2a-b}(\ov{X}, \Lambda(a)),
\]
where we have replaced the cohomology with compact support by the
usual motivic cohomology on the two end terms because $\ov{X}$ and
$Z$ are projective over $k$. It follows from the first part of the
lemma that the two end terms vanish. The desired assertion now
follows.  
\end{proof}

\begin{lem}\label{lem:coh-inj}
Let $f: W \to X$ be a finite and flat morphism of degree $p^r$ in $\Sch_k$ with
$\dim(X) = d$. Let
$Y \subseteq X$ be a closed subscheme and $Z = Y \times_X W$. Then the
pull-back map $f^*: H^{2a-b}_Y(X, \Lambda(a)) \to H^{2a-b}_Z(W, \Lambda(a))$
is injective  whenever $a > d+b$.
\end{lem}
\begin{proof}
When $Y$ (and hence $Z$) is empty, both sides are zero by 
\cite[Theorem~5.1]{KP} and so the lemma holds. 
Otherwise, we let
$U = X \setminus Y$ and $V = f^{-1}(U)$. 
Let $v: X \to \Spec(k)$ and $u: W \to \Spec(k)$ denote the structure maps.

By \cite[Definition~4.3.1,
Corollary~5.2.4]{Kelly}, the spectrum $H\Lambda \in \sS\sH$ has the structure
of traces. In particular, for any $m, n \in \Z$,
there exists a trace map ${\rm Tr}_f:
{\bf R}f_* {\bf L}f^* {\bf L}v^*(\Sigma^{m,n} H\Lambda) \to  
{\bf L}v^*(\Sigma^{m,n} H\Lambda)$ in $\stablehomotopy _X$
such that its composition with the unit of adjunction 
$(\mathbf L f^\ast, \mathbf R f_\ast)$:
\[
{\bf L}v^*(\Sigma^{m,n} H\Lambda) \rightarrow 
{\bf R}f_* {\bf L}f^* {\bf L}v^*(\Sigma^{m,n} H\Lambda) \xrightarrow{{\rm Tr}_f} 
{\bf L}v^*(\Sigma^{m,n} H\Lambda)
\]
is multiplication by $p^r$. 
In particular, the composite map
\begin{equation}\label{eqn:coh-inj-0}
\xymatrix@C.8pc{
\Hom_{\stablehomotopy_X}({\mathbf 1}_X, {\bf L}v^*(\Sigma^{m,n} H\Lambda)) 
\ar[r]^-{f^*} \ar[dr] &
\Hom_{\stablehomotopy_X}({\mathbf 1}_X, 
{\bf R}f_* {\bf L}f^* {\bf L}v^*(\Sigma^{m,n} H\Lambda)) \ar[d]^-{f_*} \\
& \Hom_{\stablehomotopy_X}({\mathbf 1}_X, {\bf L}v^*(\Sigma^{m,n} H\Lambda))}
\end{equation}
is multiplication by $p^r$, where we let $f_*$ denote the
maps on the cohomology groups induced by ${\rm Tr}_f$.
It follows that this composite map is an isomorphism.

On the other hand, we have
$\Hom_{\stablehomotopy_X}({\mathbf 1}_X, {\bf L}v^*(\Sigma^{m,n} H\Lambda)) 
\cong H^m(X, \Lambda(n))$ by \cite[Corollary~3.6]{KP}.
Moreover,
\[
\begin{array}{lll}
\Hom_{\stablehomotopy_X}({\mathbf 1}_X, 
{\bf R}f_* {\bf L}f^* {\bf L}v^*(\Sigma^{m,n} H\Lambda)) & {\cong}^1 &
\Hom_{\stablehomotopy_W}({\bf L}f^*({\mathbf 1}_X), 
{\bf L}f^* {\bf L}v^*(\Sigma^{m,n} H\Lambda)) \\
& \cong & 
\Hom_{\stablehomotopy_W}({\mathbf 1}_W, 
{\bf L}(v \circ f)^*(\Sigma^{m,n} H\Lambda)) \\
& \cong &
\Hom_{\stablehomotopy_W}({\mathbf 1}_W, 
{\bf L}u^*(\Sigma^{m,n} H\Lambda)) \\
& {\cong}^2 &
H^m(W, \Lambda(n)),
\end{array}
\]
where ${\cong}^1$ is a consequence of adjointness of the pair 
$({\bf L}f^*, {\bf R}f_*)$ and ${\cong}^2$
follows from \cite[Corollary~3.6]{KP}.
Using these isomorphisms in ~\eqref{eqn:coh-inj-0}, we get the maps
\[
H^m(X, \Lambda(n)) \xrightarrow{f^*} H^m(W, \Lambda(n)) \xrightarrow{f_*}
H^m(X, \Lambda(n))
\]
whose composite is an isomorphism.

We now consider the commutative diagram
\begin{equation}\label{eqn:coh-inj-1}
\xymatrix@C.8pc{
H^{2a-b-1}(X, \Lambda(a)) \ar[r] \ar[d]_-{f^*} & 
H^{2a-b-1}(U, \Lambda(a)) \ar[r] \ar[d]^-{f^*} &
H^{2a-b}_Y(X, \Lambda(a)) \ar[r] \ar[d]^-{f^*} & 0 \\
H^{2a-b-1}(W, \Lambda(a)) \ar[r] \ar[d]_-{f_*} & 
H^{2a-b-1}(V, \Lambda(a)) \ar[r] \ar[d]^-{f_*} &
H^{2a-b}_Z(W, \Lambda(a)) \ar[r] \ar@{.>}[d] & 0 \\
H^{2a-b-1}(X, \Lambda(a)) \ar[r]  & 
H^{2a-b-1}(U, \Lambda(a)) \ar[r] &
H^{2a-b}_Y(X, \Lambda(a)) \ar[r]  & 0,}
\end{equation}
whose rows are exact because 
$H^{2a-b}(X, \Lambda(a)) = H^{2a-b}(W, \Lambda(a)) = 0$, as we saw in the 
beginning of the proof.
It follows that there exists an arrow $f_*:
H^{2a-b}_Z(W, \Lambda(a)) \to H^{2a-b}_Y(X, \Lambda(a))$ such that
all squares in this diagram commute.
Since the left and the middle composite vertical arrows are isomorphisms,
as we just observed above, it follows that the right composite vertical arrow
must also be an isomorphism.
In particular, $f^*: H^{2a-b}_Y(X, \Lambda(a)) \to 
H^{2a-b}_Z(W, \Lambda(a))$ is injective. This finishes the proof.
\end{proof}

\begin{thm}\label{thm:Conn-KGL}
With the notation of \ref{def.relKH}, assume that $f$ is either a closed 
immersion such that $\dim(Y) < \dim(X) = d$, or an open immersion.
Then the map
$\KGL^0_X [\tfrac{1}{p}]\to s_0\KGL_X [\tfrac{1}{p}] \cong 
H\Z [\tfrac{1}{p}]$ induces for every integer $i \ge 0$, an isomorphism
\begin{equation}\label{eqn:Conn-KGL-0}
CKH^{2d+i, d+i}(M_f)_{\Lambda} \xrightarrow{\cong} H^{2d+i}(M_f, \Lambda(d+i)).
\end{equation}
In particular, the canonical map $\KGL^0_X \to \KGL_X$ 
induces a natural cycle class map
\begin{equation}\label{eqn:Conn-cycle}
cyc_i: H^{2d+i}(M_f, \Lambda(d+i)) \to KH_i(f)_{\Lambda}.
\end{equation}
\end{thm}
\begin{proof}
By \cite[5.12-5.13]{KP}, the argument in \thmref{thm:KH-SS} also applies 
to $\mathbf L \pi ^{\ast} \KGL^0$.  In particular, there is a strongly 
convergent
spectral sequence:
\begin{equation}\label{eqn:Conn-KGL-1}
E^{a,b}_2 = H^{a-b}(M_f, \Lambda_{b\le0}(n-b))
\Rightarrow CKH^{a+b, n}(M_f)_{\Lambda},
\end{equation}
where $\Lambda_{b\le0} = \Lambda$ if $b \le  0$ and is zero otherwise.
Furthermore, this spectral sequence degenerates with rational coefficients.

Combining the spectral sequence \eqref{eqn:Conn-KGL-1} with
\lemref{lem:vanish-coh}, we conclude that the
edge map $CKH^{2d+i, d+i}(M_f)_{\Lambda} \to H^{2d+i}(M_f, \Lambda(d+i))$ is an 
isomorphism for every $i \ge 0$.
Finally, we compose the inverse of this isomorphism with the canonical map
$CKH^{2d+i, d+i}(M_f)_{\Lambda} \to KH_{i}(M_f)_{\Lambda}$ to get the desired cycle 
class map.
\end{proof}

Note that the proofs of 
\lemref{lem:vanish-coh} and \thmref{thm:Conn-KGL} show that they
remain valid without
inverting $p$ and without assuming resolution of singularities if 
$f$ is a closed immersion of smooth schemes. We thus get the following.

\begin{cor}\label{cor:Conn-KGL-sm}
Let $k$ be a perfect field and let $f: Y \inj X$ be a closed immersion
in $\Sm_k$. Then there exists a cycle class map for $i \ge 0$:
\[
cyc_i: H^{2d+i}(X|Y, \Z(d+i)) \to K_{i}(X,Y)
\]
whose kernel is a torsion group.
\end{cor}

\subsection{The Chern class map to relative motivic cohomology}
\label{sec:CC}
The Chern class maps from the homotopy invariant $K$-theory of
singular schemes to their motivic cohomology was constructed in
characteristic zero in \cite[\S~6]{KP}. Using the recent result
of Kerz, Strunk and Tamme \cite{KST} and the double construction
of \cite{BK}, we can generalize the construction of
\cite{KP} to the relative setting and positive 
characteristic as follows.

\begin{thm}\label{thm:CCR}
Let $k$ be a perfect field.
Let $X$ be a smooth scheme over $k$ and let $f: Y \inj X$ be the 
inclusion of an effective Cartier divisor. 
Assume that $k$ admits resolution of singularities.
Then there are Chern class maps
\[
c_{X,Y, b}: KH_a(X,Y) \to H^{2b-a}(X|Y, \Z(b))
\]
which are functorial in the pair $(X,Y)$.
\end{thm}
\begin{proof}
To prove the theorem, we shall use the doubling trick of \cite{BK}.
We let $S_X = S(X,Y)$ denote the double of $X$ along $Y$
(see \S~\ref{sec:double}).
Using the $cdh$-descent for the motivic cohomology and $KH$-theory,
it suffices to construct the Chern class map from
$KH_a(S_X, X_-)$ to $H^{2b-a}(S_X|X_-, \Z(b))$
(see the proof of \lemref{lem:hyper-0}).
 
After this reduction, the proof is now identical to that of 
\cite[Theorem~6.9]{KP} with very
minor modifications that we explain.
We follow the notations of \cite[\S~6]{KP}.
We only have to show that \cite[Lemma~6.4]{KP} is valid in the present
case too. But this is immediate from our assumption that $k$ admits
resolution of singularities and the recent result of \cite{KST} that
the canonical map $KH(Z) \to \wt{\sK}_{cdh}(Z)$ (induced by the
$cdh$-descent for $KH$, shown in \cite{Cisinski}) is a weak equivalence of
spectra for any $Z \in \Sch_k$.

Following the rest of the argument of \cite[Theorem~6.9]{KP} verbatim,
we obtain a commutative diagram of Chern class maps
\begin{equation}\label{eqn:Comp-1} 
\xymatrix@C.8pc{
KH_a(S_X) \ar[r]^-{C_b} \ar[d]_{\iota^*_+}
& \H^{-a}_{cdh}(S_X, C_*z_{equi}(\A^b_k, 0)_{cdh}) \ar[d]_{\iota^*_+} 
\ar[r]^-{\cong} & H^{2b-a}(S_X, \Z(b)) \ar[d]^{\iota^*_+} \\
KH_a(X_-) \ar[r]^-{C_b} 
& \H^{-a}_{cdh}(X_-, C_*z_{equi}(\A^b_k, 0)_{cdh}) 
\ar[r]^-{\cong} & H^{2b-a}(X_-, \Z(b)).}
\end{equation}

Since $\iota_+: X_- \inj S_X$ is naturally split by $\nabla: S_X \to X$, we 
get the desired natural map
\begin{equation}\label{eqn:Comp-2} 
c_{X,Y, b}: KH_a(S_X, X_-) \to H^{2b-a}(S_X|X_-, \Z(b)).
\end{equation}

The functoriality of $c_{X,Y, b}$ is clear from its construction.
We refer to the proof of \cite[Theorem~6.9]{KP} for further detail.
\end{proof}

\section{Chow group with modulus and relative motivic 
cohomology}\label{sec:Chow-mot-coh}
We keep the assumption on the ground field $k$ and the coefficient ring
$\Lambda$ as before. 
In this section, we shall construct a natural map from the Chow groups with
modulus to the relative motivic cohomology groups. We shall later show
that this map is an isomorphism for 0-cycles on affine schemes.

\subsection{Higher Chow groups and motivic cohomology with 
support}\label{sec:HVG-MCCS}
Consider a Cartesian square of quasi-projective schemes 
\begin{equation}\label{eqn:MCCS-0}
\begin{array}{c}
\xymatrix@C1pc@R=1pc{
Z \ar[r]^-{u} \ar[d]_{v'} & Z' \ar[d]^{w} \\
X \ar[r]^-{v} & X'}
\end{array}
\end{equation}
where $(X',X)$ is a smooth pair with $\dim(X') = d \ge 1$,
and $X$ is a divisor. We assume that the
vertical arrows are closed immersions of codimension $q \ge 1$
such that $Z'$ is integral and is not contained in $X$.   
In particular, ~\eqref{eqn:MCCS-0} is a transverse square.
Let $\sL$ be the line bundle on $X'$ associated to $X$.
For any locally closed subscheme 
$U' \subset X'$,
let $z^i_X(U', \bullet)$ be the subcomplex of Bloch's cycle complex
$z^i(U', \bullet)$ generated by integral cycles which intersect
$X$ properly. Given any open subset $U' \subset X'$ and $U = X \cap U'$, 
there is a commutative diagram 
\begin{equation}\label{eqn:MCCS-1}
\begin{array}{c}
\xymatrix@C=.8pc@R=1pc{ 
z^i_X(Z'\cap U', \bullet) \ar[r] \ar[d] & z^{i+q}_X(U', \bullet) \ar[r] \ar[d] &
z^{i+q}_X(U' \setminus Z', \bullet) \ar[d] \\
z^i(Z'\cap U', \bullet) \ar[r] & z^{i+q}(U', \bullet) \ar[r] &
z^{i+q}(U' \setminus Z', \bullet).}
\end{array}
\end{equation}  

The localization theorem for Bloch's complex says that the bottom row
is a distinguished triangle in the derived category of chain complexes of 
abelian groups. Since $X \cap (X' \setminus Z') \neq \emptyset$
and $X \subset X'$ is a divisor, the proof of Bloch's localization
theorem shows easily that the top row is also a distinguished triangle.
By \cite[Theorem 1.10]{Krishna-Levine} it
follow that the middle and the right vertical
arrows are quasi-isomorphisms. So the left vertical arrow
is also a quasi-isomorphism.

On the other hand, there is a commutative diagram of cycle complexes
\begin{equation}\label{eqn:MCCS-2}
\begin{array}{c}
\xymatrix@C=.8pc@R=1pc{ 
z^i_X(Z'\cap U', \bullet) \ar[r] \ar[d]_{u^*} & z^{i+q}_X(U', \bullet) \ar[r] 
\ar[d]^{v^*} &
z^{i+q}_X(U' \setminus Z', \bullet) \ar[d]^{v^*} \\
z^i(Z\cap U, \bullet) \ar[r] & z^{i+q}(U, \bullet) \ar[r] &
z^{i+q}(U \setminus Z, \bullet).}
\end{array}
\end{equation}  

In particular, there is a pull-back map
$u^*: \CH^i(Z',j) \to \CH^i(Z,j)$ which is induced by capping
with the first Chern class $c_1(\sL)$.
Furthermore, it is immediate that ~\eqref{eqn:MCCS-1} and ~\eqref{eqn:MCCS-2}
are compatible with respect to the inclusions of open subsets 
$U'_1 \subset U'_2$ in $X'$. By varying the open $U' \subset X'$,
it follows that ~\eqref{eqn:MCCS-1} and ~\eqref{eqn:MCCS-2} form commutative 
diagrams of complexes of presheaves of abelian groups on the small 
Zariski site of $X'$. 

It is shown in \cite[Proposition~4.2.9]{Voev-1} that there is a
monomorphism of chain complexes $\psi_{U'}: C_{\ast}z_{equi}(U',i) \inj
z^{d-i}(U', \bullet)$ which is functorial with respect to flat pull-back and
proper push-forward. In particular, this induces a monomorphism of 
complexes of presheaves of abelian groups on the small 
Zariski site of $X'$. Note here that by $C_{\ast}z_{equi}(U',i)$,
we mean $C_{\ast}z_{equi}(U',i)(\Spec(k))$. 
For any complex $\sC$ of presheaves of abelian groups on 
the small Zariski site of $X'$, we let $\sC_{\zar}$ denote the sheafification
of $\sC$.

We get a sequence of morphisms of complexes of Zariski sheaves
\begin{equation}\label{eqn:MCCS-3}
\begin{array}{c}
\xymatrix@C=.8pc@R=1pc{ 
C_{\ast}z_{equi}(-,i) \ar[r]  &  z^{d-i}(-, \bullet)_{\zar} &
z^{d-i}_X(-, \bullet)_{\zar} \ar[l] \ar[d]^{\cap X} \\
& C_{\ast}z_{equi}(-,i-1)|_X \ar[r] & z^{d-i}(-, \bullet)_{\zar}|_X  .}
\end{array}
\end{equation}

Note that $C_{\ast}z_{equi}(-,i)$ is already a complex of Zariski sheaves.
All arrows except the vertical one in this diagram are quasi-isomorphisms
of complexes of Zariski sheaves (see \cite[Proposition~4.2.9]{Voev-1}).
It follows that there is a morphism 
$v^*: C_{\ast}z_{equi}(-,i) \to C_{\ast}z_{equi}(-,i-1)|_X$ in the derived
category of complexes of Zariski sheaves and a commutative diagram
\begin{equation}\label{eqn:MCCS-4}
\begin{array}{c}
\xymatrix@C=.8pc@R=1pc{ 
C_{\ast}z_{equi}(-,i)|_{Z'} \ar[r]  & C_{\ast}z_{equi}(-,i)|_{X'} \ar[r] 
\ar[d]^{v^*} &
C_{\ast}z_{equi}(-,i)|_{X' \setminus Z'} \ar[d]^{v^*} \\
C_{\ast}z_{equi}(-,i-1)|_{Z} \ar[r] & C_{\ast}z_{equi}(-,i-1)|_{X} \ar[r] &
C_{\ast}z_{equi}(-,i-1)|_{X\setminus Z}.}
\end{array}
\end{equation}  

It follows from ~\eqref{eqn:MCCS-2} and ~\eqref{eqn:MCCS-3} that 
the two rows in ~\eqref{eqn:MCCS-4} are distinguished triangles in
the derived category of complexes of Zariski sheaves and the
right square commutes. We therefore obtain a morphism
$u^*: C_{\ast}z_{equi}(-,i)|_{Z'} \to C_{\ast}z_{equi}(-,i-1)|_{Z}$ such that
~\eqref{eqn:MCCS-4} commutes. Moreover,  ~\eqref{eqn:MCCS-3} implies that
\begin{equation}\label{eqn:MCCS-5}
\begin{array}{c}
\xymatrix@C=.8pc@R=1pc{ 
C_{\ast}z_{equi}(-,i)|_{Z'} \ar[r] \ar[d]_{u^*} &  z^{d-i-q}(-, \bullet)_{\zar}|_{Z'}
\ar[d]^{u^*} \\
C_{\ast}z_{equi}(-,i-1)|_{Z} \ar[r] & z^{d-i-q}(-, \bullet)_{\zar}|_{Z}}
\end{array}
\end{equation}
commutes.

It follows from \cite[Proposition~4.2.9, Theorem~4.3.7]{Voev-1} 
(see also \cite[Theorem~19.11]{MVW}) that
for any $W \in \Sm_k$, the Zariski hypercohomology of 
$C_{\ast}z_{equi}(-,i)|_{W}$
are the motivic cohomology groups of $W$. In particular,
the commutative square of hypercohomology groups induced by the
right square in ~\eqref{eqn:MCCS-4} is isomorphic to the
one induced on the motivic cohomology by the commutative diagram of
motives
\begin{equation}\label{eqn:MCCS-6}
\begin{array}{c}
\xymatrix@C=.8pc@R=1pc{ 
M(X \setminus Z) \ar[r] \ar[d]_{v} & M(X) \ar[d]^{v} \\
M(X' \setminus Z') \ar[r] & M(X').}
\end{array}
\end{equation}

We conclude from  ~\eqref{eqn:MCCS-4} that the hypercohomology of
$C_{\ast}z_{equi}(-,i)|_{Z'}$ and $C_{\ast}z_{equi}(-,i)|_{Z}$ are the motivic 
cohomology groups of $M_{Z'}(X')$ and $M_Z(X)$, respectively.
Moreover, the left vertical arrow in ~\eqref{eqn:MCCS-5} is the one
induced by the canonical map $u: M_Z(X) \to M_{Z'}(X')$.
We have thus shown the following. 

\begin{lem}\label{lem:Chow-mot-comp}
Given the commutative diagram ~\eqref{eqn:MCCS-0},
there are induced maps of motives $v: M_Z(X) \to M_{Z'}(X')$ and 
cycle complexes $u^*: z^{i}(Z', \bullet) \to z^{i}(Z, \bullet)$,
and a commutative diagram of associated hypercohomology groups
\begin{equation}\label{eqn:MCCS-7}
\begin{array}{c}
\xymatrix@R=1pc{ 
\Hom_{\DM_k}(M_{Z'}(X'), \Z(i)[j]) \ar[d]_{v^*} \ar[r]^-{\alpha_{Z'}} &
\CH^{i-q}(Z', 2i-j) \ar[d]^{u^*} \\
\Hom_{\DM_k}(M_{Z}(X), \Z(i)[j]) \ar[r]^-{\alpha_{Z}} &
\CH^{i-q}(Z, 2i-j)}
\end{array}
\end{equation}
in which the horizontal arrows are isomorphisms.
\end{lem}

\subsection{Relation between Chow group and motivic cohomology}
\label{sec:Relation}
Let $X$ be a smooth quasi-projective scheme of dimension $d \ge 0$ over 
$k$ and let $D \subset X$ be an effective Cartier divisor.
We let $U = X \setminus D$. We let $S_X$ denote the double of $X$ along $D$
(see \S~\ref{sec:double}).
We have the inclusions $U \cong U_\pm \subset X_\pm \subset S_X$.

Let $q \ge 1$ be an integer and let $Z \subset X$ be an integral
cycle in $z^q(X|D, 0)$. The modulus condition implies that $Z \subset U$.
We consider the embeddings $Z \xrightarrow{i} U_+ =
S_X \setminus X_- \xrightarrow{j_+} S_X$, 
where the composite map is a closed immersion. 
We therefore have a sequence of maps

\begin{equation}\label{eqn:cycl-map}
\begin{array}{c}
\xymatrix@C=.6pc@R=1pc{
\CH^0(Z)_{\Lambda} \ar[d]^{\cong}_{\alpha^{-1}_{Z}}  
\ar@/_7pc/[dd]_(.25){\lambda_Z} & \\
\Hom_{\DML}(M_Z(U_+), \Lambda(q)[2q])  \ar[r]^{\cong} &
\Hom_{\DML}(M_Z(S_X), \Lambda(q)[2q]) \ar[d] \\
H^{2q}(S_X, \Lambda(q)) & \Hom_{\DML}(M(S_X), \Lambda(q)[2q]) \ar[l]_-{\cong}.}
\end{array}
\end{equation}

We let $\lambda_Z: \CH^0(Z)_{\Lambda} \to
H^{2q}(S_X, \Lambda(q))$ denote the composite of all arrows
in ~\eqref{eqn:cycl-map} and set $\lambda_{X|D}([Z]) = \Lambda_Z(1)$.
We extend it linearly to define a
group homomorphism $\lambda_{X|D}: z^q(X|D,0) \to H^{2q}(S_X, \Lambda(q))$.

\begin{lem}\label{lem:cycl-map*}
The map $\lambda_{X|D}$ induces a group homomorphism
\[
\lambda_{X|D}: \CH^q(X|D)_{\Lambda} \to H^{2q}(S_X, \Lambda(q)).
\]
\end{lem}
\begin{proof}
To prove the lemma, we consider the diagram
\begin{equation}\label{eqn:cycl-map-0}
\begin{array}{c}
\xymatrix@R=1pc{
z^q(X|D, 1) \ar@{^{(}->}[r] \ar[d]_{\partial^*_1 - \partial^*_0} &
z^q(\A^1_X|\A^1_D,0) \ar[r]_-{\lambda_{{\A^1_X}|{\A^1_D}}} &
H^{2q}(S_{\A^1_X}, \Lambda(q)) \ar[d]^{\iota^*_1 - \iota^*_0} \\
z^q(X|D, 0) \ar[rr]_-{\lambda_{X|D}} & &
H^{2q}(S_{X}, \Lambda(q)).}
\end{array}
\end{equation}

To show that $\Lambda_{X|D}$ kills the 
subgroup of cycles rationally equivalent to zero is equivalent to showing
that $\lambda_{X|D} \circ (\partial^*_1 - \partial^*_0) = 0$.
It follows from \cite[Proposition~2.3]{BK} that $S_{\A^1_X}$ is canonically
isomorphic to $\A^1_{S_X}$ and the right vertical arrow in
~\eqref{eqn:cycl-map-0} is induced by the inclusion
$\iota_t: S_X \inj \A^1_{S_X}$ for $t = 0,1$. It follows from the homotopy
invariance of the motivic cohomology that the right vertical arrow in
~\eqref{eqn:cycl-map-0} is zero. Our assertion will therefore follow
if we show that ~\eqref{eqn:cycl-map-0} commutes.

Let $W \in z^q(X|D, 1)$ be an integral cycle which intersects
the faces of $\A^1_X$ properly and whose closure in $\P^1_X$ satisfies 
the modulus $D$. To show the commutativity of ~\eqref{eqn:cycl-map-0}, it 
suffices to show that the diagram
\begin{equation}\label{eqn:cycl-map-1}
\begin{array}{c}
\xymatrix@C.8pc@R=1pc{
& \stackrel{r_t}{\underset{i =1}\oplus} \CH^0(Z_i)_{\Lambda} \ar[d]^{\cong} \\
\CH^0(W)_{\Lambda} \ar[d]_{\cong}  \ar[r]^-{\iota^*_t} &
\CH^0(W_t)_{\Lambda}  \ar[d]^{\cong}  \\
\Hom_{\DML}(M_W(\A^1_{U_+}), \Lambda(q)[2q])  \ar[d]_{\cong} \ar[r]^-{\iota^*_t} &
\Hom_{\DML}(M_{W_t}(U_+), \Lambda(q)[2q])  \ar[d]^{\cong} \\
\Hom_{\DML}(M_W(\A^1_{S_X}), \Lambda(q)[2q])  \ar[d] \ar[r]^-{\iota^*_t} &
\Hom_{\DML}(M_{W_t}(S_X), \Lambda(q)[2q])  \ar[d] \\
\Hom_{\DML}(M(\A^1_{S_X}), \Lambda(q)[2q])  \ar[d]_{\cong} \ar[r]^-{\iota^*_t} & 
\Hom_{\DML}(M(S_X), \Lambda(q)[2q])  \ar[d]^{\cong} \\
H^{2q}(\A^1_{S_X}, \Lambda(q)) \ar[r]^-{\iota^*_t} & 
H^{2q}(S_X, \Lambda(q))}
\end{array}
\end{equation}  
commutes for $t = 0,1$, where
$\{Z_1, \cdots , Z_{r_t}\}$ are the irreducible components of $W_t =
\iota^*_t(W)$. Note that $\iota^*_t(1)$ is the cycle class of
$W_t = \partial^*_t([W])$ in $\CH^0(W_t)_{\Lambda}  \cong \Lambda^{r_t}$.

The bottom three squares commute by the commutativity of the 
diagram
\begin{equation}\label{eqn:cycl-map-2}
\begin{array}{c}
\xymatrix@C.8pc@R=1pc{
M(S_X) \ar[r] \ar[d]_{\iota_t} & M_{W_t}(S_X)  \ar[d]^{\iota_t} & 
M_{W_t}(U_+) \ar[d]^{\iota_t} \ar[l]_-{\cong} \\
M(\A^1_{S_X}) \ar[r] & M_W(\A^1_{S_X}) & M_W(\A^1_{U_+}) \ar[l]_-{\cong}.}
\end{array}
\end{equation}

The top square commutes by \lemref{lem:Chow-mot-comp}. We conclude that
~\eqref{eqn:cycl-map-1} commutes and this completes the proof.
\end{proof}

Our main result on the relation between cycles with modulus and 
relative motivic cohomology is the following.

\begin{thm}\label{thm:Chow-mot-main}
Let $k$ be a perfect field.  
Let $X$ be a smooth quasi-projective scheme of dimension $d \ge 0$ over $k$
and let $D \subset X$ be an effective Cartier divisor. Then the following hold 
for every integer $q \ge 0$. 
\begin{enumerate}
\item
The map $\lambda_{X|D}: z^q(X|D, 0) \to H^{2q}(S_X, \Lambda(q))$
induces a map
\begin{equation}\label{eqn:cycl-map-3}
\lambda_{X|D}: \CH^q(X|D)_\Lambda \to H^{2q}(X|D, \Lambda(q)).
\end{equation}
\item
If $D \in \Sm_k$, then the inclusion $z^q(X|D, \bullet) \inj
z^q(X, \bullet)$ induces, for all $i \ge 0$, a map
\begin{equation}\label{eqn:cycl-map-3-high}
\lambda_{X|D}: \CH^q(X|D, i) \to H^{2q-i}(X|D, \Z(q)).
\end{equation}
\end{enumerate}
\end{thm}
\begin{proof}
We have a commutative diagram with exact bottom row
\eqref{eqn:nabla*}, \eqref{prop:rel-mot-coh-equid}:
\begin{equation}\label{eqn:cycl-map-4}
\begin{array}{c}
\xymatrix@C.8pc{
& & \CH^q(X|D)_\Lambda \ar[d]^{\lambda_{X|D}} \ar@{.>}[dl] & & \\
0 \ar[r] & H^{2q}(X|D, \Lambda(q)) \ar[r]^-{\nu_{+,*}} & 
H^{2q}(S_X, \Lambda(q)) \ar[r]^-{\iota^*_-} & H^{2q}(X, \Lambda(q)) \ar[r] & 0.}
\end{array}
\end{equation}

It is therefore enough to show that $\iota^*_- \circ \lambda_{X|D} = 0$.
But this is clear from the construction of $\lambda_{X|D}$ in
~\eqref{eqn:cycl-map} because $M_Z(X_-) = 0$ and the diagram
\begin{equation}\label{eqn:cycl-map-5}
\begin{array}{c}
\xymatrix@C.8pc@R=1pc{
M(X_-) \ar[r]^-{\iota_-} \ar[d] & M(S_X) \ar[d] & M(U_+) \ar[d] 
\ar[l] \\
M_Z(X_-) \ar[r]^-{\iota_-}  & M_Z(S_X) & M_Z(U_+) \ar[l]_-{\cong}}
\end{array}
\end{equation}
commutes for any integral cycle $Z \in z^q(X|D, 0)$.  
This proves (1).

We now prove (2).  Using \thmref{thm:sm-pair},
it is enough to construct a map
$\CH^q(X|D, i) \to \CH^q_{\sM}(X|D, i)$.
But this follows from the observation that 
\[
z^q_{\sM}(X|D, \bullet) \to z^q_D(X, \bullet) \to  z^q(D, \bullet) \to
z^q_{\sM}(X|D, \bullet)[1]
\]
is a distinguished triangle in the derived category of abelian groups
and the composite map $z^q(X|D, \bullet) \to z^q_D(X, \bullet) \to  
z^q(D, \bullet)$ is zero. Hence, the inclusion
$z^q(X|D, \bullet) \inj z^q_D(X, \bullet)$ factors through a
map $\lambda_{X|D}: z^q(X|D, \bullet) \to  z^q_{\sM}(X|D, \bullet)$ in the
derived category. In particular, it induces the desired map
between the homology groups.
\end{proof}

\subsection{Cycle class map for Chow groups with modulus}
\label{sec:CCMM*}
If the relative $K$-theory is to be described by higher Chow
groups with modulus, as conjectured, then there must exist a cycle
class map from the higher 0-cycles with modulus to the relative 
higher $K$-groups. As a consequence of \corref{cor:Conn-KGL-sm},
\thmref{thm:Chow-mot-main}, and the weak equivalence of spectra  
$K(X,D) \to KH(X,D)$, it follows immediately that this is indeed the case
if $D \subset X$ is a smooth divisor.

\begin{cor}\label{cor:cyle-class-KB}
Let $k$ be a perfect field and let $D \subset X$ be an inclusion of
a smooth divisor in $\Sm_k$. Then there exists a cycle class map
for $i \ge 0$:
\[
cyc_i: \CH^{d+i}(X|D, i) \to K_i(X,D).
\]
\end{cor}

A cycle class map of the kind given in \corref{cor:cyle-class-KB} was
constructed in \cite{Binda}. In that construction, Binda uses
a different definition for the Chow groups with modulus compared to the
ones described in \S~\ref{sec:HCGM}. His modulus condition is stronger and this
allows him to prove the result when $D$ is a simple normal crossing divisor.
A very general cycle class map for our definition of the Chow groups with 
modulus is constructed in \cite{GK} 
(where $D$ is allowed to be any Cartier divisor). However, this construction
exists only in the pro-setting where we need to consider the 
pro-abelian groups $\{\CH^{d+i}(X|mD, i)\}_{m \ge 1}$. 
The main point of the new result \corref{cor:cyle-class-KB} is that it
shows that for smooth pairs, we do not need to go to the pro-setting.


\section{The isomorphism theorem}\label{sec:Iso}
The goal of this section is to show that the comparison map
$\lambda_{X|D}$ of \thmref{thm:Chow-mot-main} is an isomorphism for
0-cycles on affine schemes. 
Our strategy for showing this is to use the doubling trick once again
and combine this with \thmref{thm:BS-main}. 
Other crucial ingredient is the Roitman torsion
theorem of \cite{Krishna-2}.

Throughout this section, we fix an algebraically closed field $k$. We
also fix a smooth quasi-projective scheme $X$ of dimension $d$ over $k$.
We let $D \subset X$ be a smooth divisor. Let $S_X = S(X,D)$ denote the
double of $X$ along $D$. We let $U = X \setminus D$ and $S_U = \nabla^{-1}(U)
= U_+ \amalg U_-$, where $\nabla: S_X \to X$ is the fold map.  

We first construct a comparison map for $S_X$.
Let $x \in S_U$ be a closed point.
We let $S = \Spec(k(x))$. Since $\iota_S: S \inj S_X$ is a local complete 
intersection (lci), it follows from \cite[Definition~2.32, Theorem~2.33]{Nav}
that there is a push-forward map $\iota_{S,*}: \Z \xrightarrow{\cong} 
H^0(S, \Lambda(0)) \to H^{2d}(S_X, \Z(d))$.
Since $(S_U, S)$ is a smooth pair, the map $\iota_*$ is induced by the
maps of motives 
\begin{equation}\label{eqn:Chow-mot*-1}
M(S_X) \to M_{S}(S_X) \xleftarrow{\cong} M_S(S_U) \xrightarrow{\cong} 
M(S)(d)[2d].
\end{equation} 

We let $\gamma_{S_X}([x]) = \iota_{S,*}(1) \in H^{2d}(S_X, \Z(d))$ and extend
this construction linearly to get a map
\begin{equation}\label{eqn:Comp-gen}
\gamma_{S_X}: \sZ_0(S_X) \to H^{2d}(S_X, \Z(d)).
\end{equation}

\begin{lem}\label{lem:Iso-section}
The map $\gamma_{S_X}$ descends to a group homomorphism
$\gamma_{S_X}:\CH_0(S_X) \to H^{2d}(S_X, \Z(d))$ such that the diagram
\begin{equation}\label{eqn:Chow-mot*-0}
\begin{array}{c}
\xymatrix@C=1pc@R=1pc{
\CH_0(S_X) \ar[r]^-{cyc_{S_X}} \ar[d]_{\gamma_{S_X}} & K_0(S_X)  
\ar[d]^{\kappa_{S_X}} \\
H^{2d}(S_X, \Z(d)) \ar[r]^-{cyc_{S_X}} & KH_0(S_X)}
\end{array}
\end{equation}
is commutative. 
\end{lem}
\begin{proof}
By the moving lemma, we can write $\CH_0(X) = {\sZ_0(U)}/{\sR_0(X/U)}$,
where $\sR_0(X/U) = \sZ_0(U) \cap \sR_0(X) \subset \sZ_0(X)$. 
We then have a diagram of split exact sequences
\begin{equation}\label{eqn:Iso-section-0}
\xymatrix@C.8pc{
0 \ar[r] &
\sZ_0(X|D) \ar[r]^-{p_{+ *}} \ar@{->>}[d] & \sZ_0(S_X) \ar[r]^{\iota^*_-} 
\ar@{->>}[d] &
\sZ_0(U) \ar[r] \ar@{->>}[d] & 0 \\
0 \ar[r] &  \CH_0(X|D) \ar[r]^-{p_{+ *}} \ar[d]_{\lambda_{X|D}} & 
\CH_0(S_X) \ar[r]^{\iota^*_-} \ar@{-->}[d]^{\gamma_{S_X}} &
\CH_0(X) \ar[r] \ar[d]^{\lambda_X} & 0 \\
0 \ar[r] & H^{2d}(X|D, \Z(d)) \ar[r]^-{p_{+ *}} &
H^{2d}(S_X, \Z(d)) \ar[r]^-{\iota^*_-} &  H^{2d}(X, \Z(d)) \ar[r]  & 0.}
\end{equation}

It is immediate from the definition of
$\lambda_{X|D}$ in ~\eqref{eqn:cycl-map} and
$\gamma_{S_X}$ in ~\eqref{eqn:Comp-gen}
that all squares in the outer diagram (ignoring the
middle row) in ~\eqref{eqn:Iso-section-0} commute. 
Let $\alpha \in \sZ_0(S_X)$ be such that it dies in $\CH_0(S_X)$.
We can uniquely write $\alpha = p_{+ *}(\alpha_1) + \nabla^*(\alpha_2)$.
Since $\iota^*_- \circ \nabla^*$ is identity, we must have that
$\alpha_1 \in \sR_0(X|D)$ and $\alpha_2 \in \sR_0(X/U) \subset \sZ_0(U)$.
But then $\alpha_1$ must die in $H^{2d}(X|D, \Z(d))$ by 
\thmref{thm:Chow-mot-main}(2), and hence it must die in
$H^{2d}(S_X, \Z(d))$. Similarly, $\alpha_2$ must die in
$H^{2d}(X, \Z(d))$. In particular, we must have 
$\gamma_{S_X}(\nabla^*(\alpha_2)) = 0$ in $H^{2d}(S_X, \Z(d))$.
We conclude that $\gamma_{S_X}(\alpha) =
\gamma_{S_X}(p_{+ *}(\alpha_1)) + \gamma_{S_X}(\nabla^*(\alpha_2)) = 0$.
This proves the first part.

To show that ~\eqref{eqn:Chow-mot*-0} commutes, we choose 
a closed point $x \in S_U$, set $S = \Spec(k(x))$
and consider the diagram
\begin{equation}\label{eqn:MC-sing-1}
\xymatrix@C1pc{
\Z \ar[r]^-{\simeq} & \CH_0(S) \ar[rr]^-{\simeq} \ar[dr]  
\ar[dd] & & H^0(S, \Z(0)) \ar[dd] \ar[dr] & \\
& & K_0(S) \ar[rr] \ar[dd] & & KH_0(S) \ar[dd] \\ 
& \CH_0(S_X) \ar[rr]^-{\gamma_{S_X}} \ar[dr] &  & 
H^{2d}(S_X, \Z(d)) \ar[dr] & \\
& & K_0(S_X) \ar[rr] & & KH_0(S_X).}
\end{equation}


To show that the bottom face of this cube commutes,
it suffices to show that its all other faces commute. 
Now, the top and front face clearly commute
and the left face commutes by the definition of the cycle class map
(see \cite[Lemma~3.12]{BK}).
The back face commutes by definition of $\gamma_{S_X}$. To show that the
right face commutes, we can break it into a diagram
\begin{equation}\label{eqn:MC-sing-2}
\xymatrix@C1pc{
H^0(S, \Z(0)) \ar[r] \ar[d] & H^{2d}_S(S_X, \Z(d)) \ar[r] \ar[d] &
H^{2d}(S_X, \Z(d)) \ar[d] \\
KH_0(S) \ar[r] & KH^S_0(S_X) \ar[r] & KH_0(S_X).}
\end{equation}

The left square clearly commutes and the right square commutes by
\thmref{thm:KH-SS*} and \thmref{thm:sm-pair}. The lemma is now proven.
\end{proof}

The final result of this paper is the following comparison theorem for
0-cycles.

\begin{thm}\label{thm:Comp-0-cycles}
Let $X$ be a smooth affine scheme of dimension $d \ge 1$ over an
algebraically closed field $k$ and let $D \subset X$ be a smooth
divisor. Then the map
\[
\lambda_{X|D} : \CH_0(X|D) \to H^{2d}(X|D, \Z(d))
\]
is an isomorphism.
\end{thm}
\begin{proof}
Since $(X,D)$ is a smooth pair, the double $S_X$ is a simple normal
crossing variety in the sense of \cite[\S~2.1]{EKW}. In particular,
it follows from \cite[Proposition~6.4]{EKW} that $\gamma_{S_X}$ is 
surjective. We remark here that the surjectivity of $\gamma _{V}$
is proven in the above cited work for an arbitrary simple
normal crossing variety $V$ if we work with $\Lambda$-coefficients
(see the end of \S~\ref{sec:review}). However, it is an elementary
checking that the proof yields this surjectivity with integral coefficients
if we let $V = S_X$. We conclude from ~\eqref{eqn:Iso-section-0}
that $\lambda_{X|D}$ is surjective. 

To show that $\lambda_{X|D}$ is injective, it suffices to show using 
~\eqref{eqn:Iso-section-0} that $\gamma_{S_X}$ is
injective. Using ~\eqref{eqn:Chow-mot*-0}, it suffices to show that
$\kappa_{S_X} \circ cyc_{S_X}$ is injective. By \cite[Corollary~6.8]{Krishna-2},
it suffices to show that $\kappa_{S_X}$ is injective. 

Since excision holds for the $K$-theory of affine schemes in degrees up to
zero (see \cite[Proposition~11.3]{BK}), and since it holds for
$KH$-theory in all degrees \cite{Cisinski}, 
there is a commutative diagram of exact sequences
\begin{equation}\label{eqn:MC-sing-3}
\xymatrix@C.8pc{
K_1(S^N_X) \ar[r] \ar[d] & K_1(D) \ar[r] \ar[d] & K_0(S_X) \ar[r] \ar[d] &
K_0(S^N_X) \ar[r] \ar[d] & K_0(D) \ar[d] \\
KH_1(S^N_X) \ar[r] & KH_1(D) \ar[r] & KH_0(S_X) \ar[r] &
KH_0(S^N_X) \ar[r]  & KH_0(D).} 
\end{equation}

Since $X$ and $D$ are regular, all vertical arrows except possibly the middle
one are isomorphisms. It follows that middle vertical arrow is also an
isomorphism. In other words, $\kappa_{S_X}$ is in fact an isomorphism. 
This finishes the proof.
\end{proof}

\begin{remk}\label{remk:Surj-gen*}
The proof of \thmref{thm:Comp-0-cycles} shows that $\lambda_{X|D}$ is
surjective even if $X$ is not affine. 
\end{remk}

\enlargethispage{60pt}

\end{document}